\documentclass[a4paper, 11pt]{amsart}  %option: draft or final
\usepackage{latexsym, amsmath, amsthm, amssymb, setspace, verbatim, bbm}
\usepackage[all]{xy}
\usepackage[utf8]{inputenc} \usepackage[T1]{fontenc} \usepackage{lmodern}
\usepackage{hyperref, aliascnt}
\usepackage{enumitem}
\usepackage{mathrsfs}

\title[The $KK$-uniqueness theorem]{The uniqueness theorem for Kasparov theory}
\author{Gábor Szabó}
\address{Department of Mathematics, KU Leuven, Celestijnenlaan 200b, box 2400\linebreak
\phantom{-}\hspace{2mm} B-3001 Leuven, Belgium}
\email{gabor.szabo@kuleuven.be}
\thanks{Funded by the European Union. Views and opinions expressed are those of the authors only and do not necessarily reflect those of the European Union or the European Research Council. Neither the EU nor the ERC can be held responsible for them.}

\numberwithin{equation}{section}

\begin{document}

% math
\renewcommand\matrix[1]{\left(\begin{array}{*{10}{c}} #1 \end{array}\right)}  % Matrix
\newcommand\set[1]{\left\{#1\right\}}  % Menge

%% Besondere Variablen
%Zahlmengen-Stil
\newcommand{\IA}[0]{\mathbb{A}} \newcommand{\IB}[0]{\mathbb{B}}
\newcommand{\IC}[0]{\mathbb{C}} \newcommand{\ID}[0]{\mathbb{D}}
\newcommand{\IE}[0]{\mathbb{E}} \newcommand{\IF}[0]{\mathbb{F}}
\newcommand{\IG}[0]{\mathbb{G}} \newcommand{\IH}[0]{\mathbb{H}}
\newcommand{\II}[0]{\mathbb{I}} \renewcommand{\IJ}[0]{\mathbb{J}}
\newcommand{\IK}[0]{\mathbb{K}} \newcommand{\IL}[0]{\mathbb{L}}
\newcommand{\IM}[0]{\mathbb{M}} \newcommand{\IN}[0]{\mathbb{N}}
\newcommand{\IO}[0]{\mathbb{O}} \newcommand{\IP}[0]{\mathbb{P}}
\newcommand{\IQ}[0]{\mathbb{Q}} \newcommand{\IR}[0]{\mathbb{R}}
\newcommand{\IS}[0]{\mathbb{S}} \newcommand{\IT}[0]{\mathbb{T}}
\newcommand{\IU}[0]{\mathbb{U}} \newcommand{\IV}[0]{\mathbb{V}}
\newcommand{\IW}[0]{\mathbb{W}} \newcommand{\IX}[0]{\mathbb{X}}
\newcommand{\IY}[0]{\mathbb{Y}} \newcommand{\IZ}[0]{\mathbb{Z}}

\newcommand{\Ia}[0]{\mathbbmss{a}} \newcommand{\Ib}[0]{\mathbbmss{b}}
\newcommand{\Ic}[0]{\mathbbmss{c}} \newcommand{\Id}[0]{\mathbbmss{d}}
\newcommand{\Ie}[0]{\mathbbmss{e}} \newcommand{\If}[0]{\mathbbmss{f}}
\newcommand{\Ig}[0]{\mathbbmss{g}} \newcommand{\Ih}[0]{\mathbbmss{h}}
\newcommand{\Ii}[0]{\mathbbmss{i}} \newcommand{\Ij}[0]{\mathbbmss{j}}
\newcommand{\Ik}[0]{\mathbbmss{k}} \newcommand{\Il}[0]{\mathbbmss{l}}
\renewcommand{\Im}[0]{\mathbbmss{m}} \newcommand{\In}[0]{\mathbbmss{n}}
\newcommand{\Io}[0]{\mathbbmss{o}} \newcommand{\Ip}[0]{\mathbbmss{p}}
\newcommand{\Iq}[0]{\mathbbmss{q}} \newcommand{\Ir}[0]{\mathbbmss{r}}
\newcommand{\Is}[0]{\mathbbmss{s}} \newcommand{\It}[0]{\mathbbmss{t}}
\newcommand{\Iu}[0]{\mathbbmss{u}} \newcommand{\Iv}[0]{\mathbbmss{v}}
\newcommand{\Iw}[0]{\mathbbmss{w}} \newcommand{\Ix}[0]{\mathbbmss{x}}
\newcommand{\Iy}[0]{\mathbbmss{y}} \newcommand{\Iz}[0]{\mathbbmss{z}}

%Geschwungener Stil
\newcommand{\CA}[0]{\mathcal{A}} \newcommand{\CB}[0]{\mathcal{B}}
\newcommand{\CC}[0]{\mathcal{C}} \newcommand{\CD}[0]{\mathcal{D}}
\newcommand{\CE}[0]{\mathcal{E}} \newcommand{\CF}[0]{\mathcal{F}}
\newcommand{\CG}[0]{\mathcal{G}} \newcommand{\CH}[0]{\mathcal{H}}
\newcommand{\CI}[0]{\mathcal{I}} \newcommand{\CJ}[0]{\mathcal{J}}
\newcommand{\CK}[0]{\mathcal{K}} \newcommand{\CL}[0]{\mathcal{L}}
\newcommand{\CM}[0]{\mathcal{M}} \newcommand{\CN}[0]{\mathcal{N}}
\newcommand{\CO}[0]{\mathcal{O}} \newcommand{\CP}[0]{\mathcal{P}}
\newcommand{\CQ}[0]{\mathcal{Q}} \newcommand{\CR}[0]{\mathcal{R}}
\newcommand{\CS}[0]{\mathcal{S}} \newcommand{\CT}[0]{\mathcal{T}}
\newcommand{\CU}[0]{\mathcal{U}} \newcommand{\CV}[0]{\mathcal{V}}
\newcommand{\CW}[0]{\mathcal{W}} \newcommand{\CX}[0]{\mathcal{X}}
\newcommand{\CY}[0]{\mathcal{Y}} \newcommand{\CZ}[0]{\mathcal{Z}}

%Script Stil
\newcommand{\FA}[0]{\mathfrak{A}} \newcommand{\FB}[0]{\mathfrak{B}}
\newcommand{\FC}[0]{\mathfrak{C}} \newcommand{\FD}[0]{\mathfrak{D}}
\newcommand{\FE}[0]{\mathfrak{E}} \newcommand{\FF}[0]{\mathfrak{F}}
\newcommand{\FG}[0]{\mathfrak{G}} \newcommand{\FH}[0]{\mathfrak{H}}
\newcommand{\FI}[0]{\mathfrak{I}} \newcommand{\FJ}[0]{\mathfrak{J}}
\newcommand{\FK}[0]{\mathfrak{K}} \newcommand{\FL}[0]{\mathfrak{L}}
\newcommand{\FM}[0]{\mathfrak{M}} \newcommand{\FN}[0]{\mathfrak{N}}
\newcommand{\FO}[0]{\mathfrak{O}} \newcommand{\FP}[0]{\mathfrak{P}}
\newcommand{\FQ}[0]{\mathfrak{Q}} \newcommand{\FR}[0]{\mathfrak{R}}
\newcommand{\FS}[0]{\mathfrak{S}} \newcommand{\FT}[0]{\mathfrak{T}}
\newcommand{\FU}[0]{\mathfrak{U}} \newcommand{\FV}[0]{\mathfrak{V}}
\newcommand{\FW}[0]{\mathfrak{W}} \newcommand{\FX}[0]{\mathfrak{X}}
\newcommand{\FY}[0]{\mathfrak{Y}} \newcommand{\FZ}[0]{\mathfrak{Z}}

\newcommand{\Fa}[0]{\mathfrak{a}} \newcommand{\Fb}[0]{\mathfrak{b}}
\newcommand{\Fc}[0]{\mathfrak{c}} \newcommand{\Fd}[0]{\mathfrak{d}}
\newcommand{\Fe}[0]{\mathfrak{e}} \newcommand{\Ff}[0]{\mathfrak{f}}
\newcommand{\Fg}[0]{\mathfrak{g}} \newcommand{\Fh}[0]{\mathfrak{h}}
\newcommand{\Fi}[0]{\mathfrak{i}} \newcommand{\Fj}[0]{\mathfrak{j}}
\newcommand{\Fk}[0]{\mathfrak{k}} \newcommand{\Fl}[0]{\mathfrak{l}}
\newcommand{\Fm}[0]{\mathfrak{m}} \newcommand{\Fn}[0]{\mathfrak{n}}
\newcommand{\Fo}[0]{\mathfrak{o}} \newcommand{\Fp}[0]{\mathfrak{p}}
\newcommand{\Fq}[0]{\mathfrak{q}} \newcommand{\Fr}[0]{\mathfrak{r}}
\newcommand{\Fs}[0]{\mathfrak{s}} \newcommand{\Ft}[0]{\mathfrak{t}}
\newcommand{\Fu}[0]{\mathfrak{u}} \newcommand{\Fv}[0]{\mathfrak{v}}
\newcommand{\Fw}[0]{\mathfrak{w}} \newcommand{\Fx}[0]{\mathfrak{x}}
\newcommand{\Fy}[0]{\mathfrak{y}} \newcommand{\Fz}[0]{\mathfrak{z}}

%mathscr style
\newcommand{\SA}[0]{\mathscr{A}} \newcommand{\SB}[0]{\mathscr{B}}
\newcommand{\SC}[0]{\mathscr{C}} \newcommand{\SD}[0]{\mathscr{D}}
\newcommand{\SE}[0]{\mathscr{E}} \newcommand{\SF}[0]{\mathscr{F}}
\newcommand{\SG}[0]{\mathscr{G}} \newcommand{\SH}[0]{\mathscr{H}}
\newcommand{\SI}[0]{\mathscr{I}} \newcommand{\SJ}[0]{\mathscr{J}}
\newcommand{\SK}[0]{\mathscr{K}} \newcommand{\SL}[0]{\mathscr{L}}
\newcommand{\SM}[0]{\mathscr{M}} \newcommand{\SN}[0]{\mathscr{N}}
\newcommand{\SO}[0]{\mathscr{O}} \newcommand{\SP}[0]{\mathscr{P}}
\newcommand{\SQ}[0]{\mathscr{Q}} \newcommand{\SR}[0]{\mathscr{R}}
\renewcommand{\SS}[0]{\mathscr{S}} \newcommand{\ST}[0]{\mathscr{T}}
\newcommand{\SU}[0]{\mathscr{U}} \newcommand{\SV}[0]{\mathscr{V}}
\newcommand{\SW}[0]{\mathscr{W}} \newcommand{\SX}[0]{\mathscr{X}}
\newcommand{\SY}[0]{\mathscr{Y}} \newcommand{\SZ}[0]{\mathscr{Z}}

%Modifikation der Variablen
\renewcommand{\phi}[0]{\varphi}
\newcommand{\eps}[0]{\varepsilon}

%zusätzliche Features
\newcommand{\id}[0]{\operatorname{id}}		% Identität
\newcommand{\eins}[0]{\mathbf{1}}			% Eine Eins in allgemeinerem Kontext, z.B. in einem Ring
\newcommand{\diag}[0]{\operatorname{diag}}
\newcommand{\ad}[0]{\operatorname{Ad}}
\newcommand{\ev}[0]{\operatorname{ev}}
\newcommand{\fin}[0]{{\subset\!\!\!\subset}}
\newcommand{\Aut}[0]{\operatorname{Aut}}
\newcommand{\dst}[0]{\displaystyle}
\newcommand{\cstar}[0]{\ensuremath{\mathrm{C}^*}}
\newcommand{\dist}[0]{\operatorname{dist}}
\newcommand{\cc}[0]{\simeq_{\mathrm{cc}}}
\newcommand{\ue}[0]{{~\approx_{\mathrm{u}}}~}
\newcommand{\prim}[0]{\ensuremath{\mathrm{Prim}}}
\newcommand{\GL}[0]{\operatorname{GL}}
\newcommand{\Hom}[0]{\operatorname{Hom}}
\newcommand{\sep}[0]{\ensuremath{\mathrm{sep}}}
\newcommand{\wc}[0]{\preccurlyeq}
\newcommand{\supp}[0]{\operatorname{supp}}
\renewcommand{\asymp}[0]{\sim_{\mathrm{asymp}}}
\newcommand{\asuc}[0]{\precsim_{\mathrm{as.u.}}}
\newcommand{\MvN}[0]{\ensuremath{\mathrm{MvN}}}
\newcommand{\nuc}[0]{\ensuremath{\mathrm{nuc}}}

% theorems
\newtheorem{satz}{Satz}[section]		% <--- optional, zählt so mit den Abschnitten

\newaliascnt{corCT}{satz}
\newtheorem{cor}[corCT]{Corollary}
\aliascntresetthe{corCT}
\providecommand*{\corCTautorefname}{Corollary}
\newaliascnt{lemmaCT}{satz}
\newtheorem{lemma}[lemmaCT]{Lemma}
\aliascntresetthe{lemmaCT}
\providecommand*{\lemmaCTautorefname}{Lemma}
\newaliascnt{propCT}{satz}
\newtheorem{prop}[propCT]{Proposition}
\aliascntresetthe{propCT}
\providecommand*{\propCTautorefname}{Proposition}
\newaliascnt{theoremCT}{satz}
\newtheorem{theorem}[theoremCT]{Theorem}
\aliascntresetthe{theoremCT}
\providecommand*{\theoremCTautorefname}{Theorem}
\newtheorem*{theoreme}{Theorem}

\theoremstyle{definition}

\newaliascnt{conjectureCT}{satz}
\newtheorem{conjecture}[conjectureCT]{Conjecture}
\aliascntresetthe{conjectureCT}
\providecommand*{\conjectureCTautorefname}{Conjecture}
\newaliascnt{defiCT}{satz}
\newtheorem{defi}[defiCT]{Definition}
\aliascntresetthe{defiCT}
\providecommand*{\defiCTautorefname}{Definition}
\newtheorem*{defie}{Definition}
\newaliascnt{notaCT}{satz}
\newtheorem{nota}[notaCT]{Notation}
\aliascntresetthe{notaCT}
\providecommand*{\notaCTautorefname}{Notation}
\newtheorem*{notae}{Notation}
\newaliascnt{remCT}{satz}
\newtheorem{rem}[remCT]{Remark}
\aliascntresetthe{remCT}
\providecommand*{\remCTautorefname}{Remark}
\newtheorem*{reme}{Remark}
\newaliascnt{exampleCT}{satz}
\newtheorem{example}[exampleCT]{Example}
\aliascntresetthe{exampleCT}
\providecommand*{\exampleCTautorefname}{Example}
\newaliascnt{questionCT}{satz}
\newtheorem{question}[questionCT]{Question}
\aliascntresetthe{questionCT}
\providecommand*{\questionCTautorefname}{Question}
\newtheorem*{questione}{Question}
\newaliascnt{conventionCT}{satz}
\newtheorem{convention}[conventionCT]{Convention}
\aliascntresetthe{conventionCT}
\providecommand*{\conventionCTautorefname}{Convention}

\newcounter{theoremintro}
\renewcommand*{\thetheoremintro}{\Alph{theoremintro}}
\newaliascnt{theoremiCT}{theoremintro}
\newtheorem{theoremi}[theoremiCT]{Theorem}
\aliascntresetthe{theoremiCT}
\providecommand*{\theoremiCTautorefname}{Theorem}
\newaliascnt{defiiCT}{theoremintro}
\newtheorem{defii}[defiiCT]{Definition}
\aliascntresetthe{defiiCT}
\providecommand*{\defiiCTautorefname}{Definition}
\newaliascnt{coriCT}{theoremintro}
\newtheorem{cori}[coriCT]{Corollary}
\aliascntresetthe{coriCT}
\providecommand*{\coriCTautorefname}{Corollary}
\newaliascnt{questioniCT}{theoremintro}
\newtheorem{questioni}[questioniCT]{Question}
\aliascntresetthe{questioniCT}
\providecommand*{\questioniCTautorefname}{Question}

%%%%%%%%%%%%%%%%%%%%%%%%%%%%%%%%%%%%%%%%%%%%

\begin{abstract} 
Answering a question of Carrión et al in their recent landmark paper on \cstar-algebra classification, we prove a general uniqueness theorem for $KK$-theory.
Given arbitrary separable \cstar-algebras $A$ and $B$ and a Cuntz pair consisting of two absorbing representations $\phi,\psi: A\to\CM(B\otimes\CK)$, the induced element of $KK(A,B)$ vanishes if and only if $\phi$ and $\psi$ are strongly asymptotically unitarily equivalent.
This improves upon the Lin--Dadarlat--Eilers stable uniqueness theorem.
The conclusion is deduced by showing the $K_1$-injectivity of an auxiliary \cstar-algebra associated to the \cstar-pair $(A,B)$, which is sometimes called the Paschke dual algebra in the literature.
Most of the article is concerned with the treatment of an umbrella theorem, which yields such a uniqueness theorem for other variants of $KK$-theory.
This encompasses nuclear $KK$-theory, ideal-related $KK$-theory, equivariant $KK$-theory, or any combinations thereof. 
\end{abstract}

\maketitle

\setcounter{tocdepth}{1}
\tableofcontents

%%%%%%%%%%%%%%%%%%%%%%%%%%%

\section*{Introduction}

Kasparov theory, also known as $KK$-theory, is a very powerful invariant for \cstar-algebras.
It is a bivariant version of topological $K$-theory, which includes ordinary $K$-theory and $K$-homology, alongside Brown--Douglas--Fillmore theory \cite{BrownDouglasFilmore77}, as special cases.
Invented by Kasparov \cite{Kasparov81, Kasparov88}, it can be viewed as either a bifunctor from (separable) C*-algebras to abelian groups, or as a way to equip the class of (separable) \cstar-algebras with a more flexible notion of homomorphisms.
Given a second-countable locally compact group $G$, Kasparov's construction also yields such a bifunctor on the category of all (separable) $G$-\cstar-algebras, which provided impactful applications on classical problems such as the Novikov conjecture.
The original approach by Kasparov, commonly referred to as the Kasparov picture, is still the
most common way of treating this subject, but there are other equivalent appproaches that are geared towards different types of applications.

In this article we are mainly interested in the so-called Cuntz picture \cite{Cuntz83} of $KK$-theory. 
In Cuntz' framework, elements in Kasparov's group $KK(A,B)$ for separable \cstar-algebras $A$ and $B$ can be described as homotopy classes of so-called \emph{Cuntz pairs}, i.e., pairs of $*$-homomorphisms $\phi, \psi: A\to\CM(B\otimes\CK)$ such that $\phi(a)-\psi(a) \in B\otimes\CK$ for all $a\in A$.
The vanishing of such a homotopy class can a priori be unraveled, roughly speaking, in terms of the existence of a particular kind of homotopy between the two representations $\phi$ and $\psi$.
In practice, one often deals with pairs of representations that are in a sense very large, which may allow for additional control over the specific form of such a homotopy.
For applications of $KK$-theory to classification theory, which serves as the main motivation for this article, the best conclusion one can hope for is that one can find a continuous path of unitaries in $(B\otimes\CK)^\dagger$ that asymptotically conjugates $\phi$ to $\psi$ in the point-norm topology.
The constraint of the unitaries being chosen in $(B\otimes\CK)^\dagger$ is essential, since finding such a unitary path in the multiplier algebra is neither good enough for most applications, nor is it even sufficient to guarantee the vanishing of the $KK$-class associated to the pair $(\phi,\psi)$ to begin with.
Answering \cite[Question 5.17]{CGSTW23} and \cite[Problem LXII]{SchafhauserTikuisisWhite25}, this article aims to confirm that this best possible outcome is always attainable as soon as a known necessary condition is satisfied.
One would refer to a statement of the form below as a ``$KK$-uniqueness theorem''.

\begin{theoremi} \label{theorem:A}
Suppose that $A$ and $B$ are separable \cstar-algebras and $B$ is stable.
Let $\phi,\psi: A\to\CM(B)$ be two absorbing representations forming a Cuntz pair.
If $[\phi,\psi]=0$ in $KK(A,B)$, then $\phi$ and $\psi$ are strongly asymptotically unitarily equivalent, i.e., there exists a unitary path $v: [0,\infty)\to\CU(\eins+B)$ with $v_0=\eins$ such that $\dst\psi(a)=\lim_{t\to\infty} v_t\phi(a)v_t^*$ for all $a\in A$.
\end{theoremi}

Prior to this, the seminal theorem concerning the analytical properties of $KK$-theory (in the Cuntz picture) has been the so-called \emph{stable uniqueness theorem} proved independently by Lin \cite{Lin02} and Dadarlat--Eilers \cite{DadarlatEilers01, DadarlatEilers02}.
The theorem states that if an arbitrary Cuntz pair $(\phi, \psi)$ represents the zero element in $KK(A,B)$, then $\phi$ and $\psi$ are asymptotically unitarily equivalent as above (without demanding $v_0=\eins$) after stabilizing both $\phi$ and $\psi$ with an \emph{absorbing} representation.
It has developed into one of the central pieces of methodology in the structure and classification theory of nuclear \cstar-algebras; see for example \cite{TikuisisWhiteWinter17, GongLinNiu20-2, GongLinNiu20, ElliottGongLinNiu20, ElliottGongLinNiu20-2, GongLin20, Schafhauser20, GongLin22, Gabe24, GongLin24, CGSTW23}.
Note that an elegant treatment of the stable uniqueness theorem was recently included in \cite{Cuntz25}.

In the context of both \autoref{theorem:A} and the prior stable uniqueness theorem, the role of the absorbing assumption is that of a necessary largeness criterion for the representations in question.
The notion of an absorbing representation --- recalled after \autoref{def:absorption} for \cstar-dynamics --- is inspired by the conclusion of Voiculescu's fundamental theorem \cite{Voiculescu76}.
Its assertion is equivalent to saying that every unitizably full\footnote{In this specific context, this becomes one of two more familiar statements. If the \cstar-algebra is non-unital, then the representation needs to be faithful and essential. If the \cstar-algebra is unital, then it also needs to have the property that the representation is supported on a subspace with infinite-dimensional orthogonal complement.} representation of a separable \cstar-algebra on the Hilbert space $\ell^2(\IN)$ is absorbing.
Absorbing representations always exist between separable \cstar-algebras \cite{Thomsen01} and it is known that all elements of a given $KK$-group can be represented by a Cuntz pair of absorbing representations.
When one is faced with applications to the classification of nuclear \cstar-algebras, one can detect such absorbing behavior via Elliott--Kucerovsky's ``pure largeness'' criterion \cite{ElliottKucerovsky01} and its known variants; see \cite[Subsection 4.2]{BouwenGabe25}.
For the purpose of this article, the primary application to have in mind in classification is the approach of Carrión et al \cite{CGSTW23}, where part of the technical work is about getting closer to a setup where one can compare a pair of naturally occurring absorbing representations via $KK$-theory.
The stable uniqueness theorem allows one to obtain a strong form of equivalence for the pair of representations, but only after first stabilizing with an auxiliary representation.
The process to ``destabilize'', so to speak, often required a substantial amount of extra work involving the specific assumptions related to the application.
Schafhauser demonstrated in \cite{Schafhauser20} how this issue can be circumvented under a UHF-stability assumption.
Building on this idea, it was shown in \cite[Theorem 5.15]{CGSTW23} that one can omit the stabilization with the auxiliary representation at the expense of adding a tensor copy of the Jiang--Su algebra $\CZ$ to the codomain.
Since said article concerns the classification of \cstar-algebras that are ultimately Jiang--Su stable, this extra tensor copy of $\CZ$ can be removed in a later phase of the application with some extra steps.
More recently, I strengthened said result with Farah in \cite[Theorem 5.5]{FarahSzabo24} by verifying that \autoref{theorem:A} holds under the extra assumption of Jiang--Su stability on $B$.
However, it remained a natural question to what extent it is generally necessary to stabilize with an auxiliary representation when one compares a pair of representations that are absorbing to begin with, to which \autoref{theorem:A} provides the answer.

If one broadens the scope of this investigation to equivariant $KK$-theory, there is an analog of Cuntz' picture involving pairs of \emph{cocycle representations} satisfying a similar condition as above; see also \cite[Section 6]{Szabo21cc}.
Similarly to the nonequivariant case, the Cuntz--Thomsen picture \cite{Thomsen98} for equivariant $KK$-theory has turned out to be very useful for the classification of \cstar-dynamics.
I have generalized the stable uniqueness theorem to the setting of \cstar-dynamics together with Gabe in \cite{GabeSzabo25}, which was one of the key conceptual ingredients behind our dynamical Kirchberg--Phillips theorem \cite{GabeSzabo24}.
%Further interesting applications of this theorem may be expected to come in the realm of \cstar-dynamics and their classification.
Its counterpart for actions of unitary tensor categories followed soon afterwards \cite{KitamuraNeaguPacheco24} and may become instrumental to classify unitary tensor category actions on Kirchberg algebras.

The discussions in the aforementioned references \cite{CGSTW23, SchafhauserTikuisisWhite25} identified the main technical obstacle towards the conclusion of \autoref{theorem:A} as the $K_1$-injectivity of the auxiliary \cstar-algebra $\CQ(B)\cap\bar{\phi}(A)'$,
which had been a problem considered beforehand in other related works \cite{Paschke81, GabeRuiz20, LoreauxNg20, LNS26} (here ``$\CQ$'' denotes the corona algebra and $\bar{\phi}$ is the induced homomorphism).
The first and main goal of this article is to prove \autoref{theorem:A}, by way of showing that the technical obstacle related to $K_1$-injectivity has a positive solution for arbitrary pairs of separable \cstar-algebras.
To describe the approach behind this, it is more useful to consider the \cstar-algebra $\SD_\phi\subseteq\CM(B)$ of all multipliers whose induced class belongs to $\CQ(B)\cap\bar{\phi}(A)'$; see \autoref{nota:D-phi}.
The main result of the second section is the following consequence of \autoref{thm:K1-injectivity} and \autoref{cor:Paschke-dual-K1-injective}.
The added layers of generality within \autoref{thm:K1-injectivity} include other recent results of this type like \cite[Theorems 2.5, 2.9]{LoreauxNg20} and \cite[Theorem 3.28]{LNS26} as special cases.

\begin{theoremi} \label{theorem:B}
Let $A$ and $B$ be separable \cstar-algebras with $B$ stable.
If $\phi: A\to\CM(B)$ is a (unitally) absorbing $*$-homomorphism, then $\SD_\phi$ and all its quotients are $K_1$-injective.
\end{theoremi}

%Combined with the detailed discussions appearing after \cite[Question 5.17]{CGSTW23} and \cite[Problem LXII]{SchafhauserTikuisisWhite25}, this implies a positive answer to \autoref{question:A}.
%The rest of the article includes a self-contained and more direct proof of the $KK$-uniqueness theorem based on \autoref{theorem:B}.

Before we move on to the remaining content of the article, I shall say a few words about the ideas leading to \autoref{theorem:B}.
The core argument behind the proof draws heavy inspiration from the proof of Cuntz--Higson's theorem \cite{CuntzHigson87} that the unitary group of every stable multiplier algebra $\CM(B)=\CM(B\otimes\CK)$ (recal that $B$ is assumed stable) is connected in the norm topology.
If one analyzes said proof carefully, one can interpret it as a two-step approach:
The first (and relatively easy) step is to observe that the $K$-groups of $\CM(B)$ vanish, which is a consequence of the fact that it possesses an \emph{infinite repeat} endomorphism, which automatically trivializes the identity map in $K$-theory.
The second (and more difficult) step is to show that $\CM(B)$ is $K_1$-injective.
Since it is properly infinite as a unital \cstar-algebra, there are known criteria to verify $K_1$-injectivity (these are nicely summarized in \cite{BlanchardRohdeRordam08}), one of which was the original criterion verified by Cuntz--Higson.
The proof of \autoref{theorem:B} (say for the $K_1$-injectivity of the \cstar-algebra $\SD_\phi$ itself) can be seen to mimic Cuntz--Higson's argument, with the added point that we need to keep track of the relative commutant condition appearing in the definition of $\SD_\phi$.
Fortunately, this works in remarkable generality as long as $\phi$ is sufficiently similar to an infinite repeat representation, which is automatically the case when $\phi$ is absorbing.
With this requisite in hand, the main difference between Cuntz--Higson's original argument and the one appearing here is that instead of utilizing arbitrary increasing approximate units in $B$, we utilize approximate units that are quasicentral relative to $\phi(A)$.
Since the proof involves no further assumptions, the conclusion about $K_1$-injectivity can be drawn in a much more flexible situation (described in \autoref{thm:K1-injectivity}), which is applicable to other auxiliary \cstar-algebras that play an analogous role for other types of $KK$-theory such as nuclear, ideal-related, or equivariant $KK$-theory.
The argument does however very much depend on the fact that the \cstar-algebra in question comes from a stable multiplier algebra.
Hence I would like to mention, for those readers wondering, that the proof behind \autoref{theorem:B} likely sheds no light on the general $K_1$-injectivity problem for properly infinite \cstar-algebras; see \cite[Question 2.9]{BlanchardRohdeRordam08} and \cite[Problem LXI]{SchafhauserTikuisisWhite25}.

In the rest of the article, I aim to treat the framework concerning the $KK$-uniqueness theorem in the most general setting possible for \cstar-algebras or \cstar-dynamics.
Having in mind other known variants of the stable uniqueness theorem and their applications, it can be just as useful to look at $KK$-groups with extra structure, for instance Skandalis' nuclear $KK$-theory \cite{Skandalis88} or Kirchberg's ideal-related $KK$-theory \cite{KirchbergC}; see \cite[Section 13]{Gabe24} for a unified treatment.
It stands to reason that in the long term, such $KK$-groups with enriched structure will also play a similar role in the structure and classification of \cstar-dynamics.
The third section introduces an umbrella framework of sorts, which to my knowledge unifies all known variants of (equivariant) $KK$-theory that have received significant attention in the literature.
Given two separable $G$-\cstar-algebras $(A,\alpha)$ and $(B,\beta)$ with $\beta$ strongly stable, the needed input data consists of some set $\FC$ of cocycle representations $(\phi,\Iu): (A,\alpha)\to (\CM(B),\beta)$ containing all elements of the form $(0,\Ix)$, and which is downward closed under \emph{weak containment} (\autoref{def:wc}).
A set $\FC$ satisfying these abstract properties is referred to as an \emph{eligible set}.
The modified $KK$-group $KK^G(\FC;\alpha,\beta)$ is then built, in analogy to \cite{Thomsen98}, from Cuntz pairs of cocycle representations in $\FC$, with a homotopy relation that requires the continuous paths to consist of cocycle representations in $\FC$ (called $\FC$-homotopy).
As is discussed further in \autoref{ex:generalized-KK-examples}, this general construction has as special cases nuclear and ideal-related $KK$-theory.
In the third section it is shown that one of Kasparov's fundamental observations extends to this variant of $KK$-theory, namely that $\FC$-homotopy is equivalent to operator homotopy after stabilizing with a representation in $\FC$.
This generalizes the results in \cite[Section 5]{GabeSzabo25} and is based on just slightly adjusting the approach therein to leverage the properties of an eligible set.

As a consequence of results in \cite{GabeSzabo25}, as long as the involved \cstar-algebras are separable, every eligible set has some absorbing element, i.e., a representation $(\phi,\Iu)\in\FC$ that absorbs every representation in $\FC$.
Just as before, it can be shown that every element of $KK^G(\FC;\alpha,\beta)$ has a representative Cuntz pair consisting of two absorbing elements in $\FC$.
In the fourth section, the results obtained so far are combined with a few further analytic arguments to prove the appropriate $KK$-uniqueness theorem for these generalized $KK$-groups.
The following is the separable version of \autoref{thm:KK-uniqueness}.

\begin{theoremi} \label{theorem:C}
Let $G$ be a second-countable locally compact group.
Let $\alpha: G\curvearrowright A$ and $\beta: G\curvearrowright B$ be two actions on separable \cstar-algebras, and suppose that $\beta$ strongly stable.
Let $\FC$ be an eligible set of cocycle representations $(A,\alpha)\to (\CM(B),\beta)$.
Let $(\phi,\Iu)$ and $(\psi,\Iv)$ be two absorbing elements of $\FC$ that form an anchored $(\FC;\alpha,\beta)$-Cuntz pair.
Then $[(\phi,\Iu),(\psi,\Iv)]=0$ in $KK^G(\FC;\alpha,\beta)$ if and only if $(\phi,\Iu)$ and $(\psi,\Iv)$ are strongly asymptotically unitarily equivalent.
\end{theoremi}

Upon applying this umbrella theorem to more familiar cases, we deduce Corollaries \ref{cor:dynamical-KK-uniqueness}, \ref{cor:nuclear-KK-uniqueness} and \ref{cor:ideal-related-KK-uniqueness}, which are variants of \autoref{theorem:A} for equivariant $KK$-theory as well as nuclear and/or ideal-related $KK$-theory.

The proof of \autoref{theorem:C} is very much entwined with a Paschke-type duality, which is more rigorously explored in the fifth section.
A variant of the known Paschke duality \cite[Theorem 3.2]{Thomsen01} asserts that if $\phi: A\to\CM(B)$ is an absorbing representation, then $KK(A,B)$ is naturally isomorphic to $K_1(\SD_\phi)$.
The auxiliary \cstar-algebra $\SD_\phi$ is hence sometimes referred to as the ``Paschke dual algebra'' of the pair $(A,B)$.
A dynamical version of this result was proved in \cite[Theorem 6.2]{Thomsen05} in the language of genuinely equivariant representations under a stronger stability assumption on $\beta$.
Thomsen's work inspires the approach given here and served as my original motivation to develop his ideas in the framework of cocycle representations.
We explore a generalization of this Paschke duality phenomenon, which comes with an extra layer of subtlety in how to encode the dynamics.
In the case of \cstar-dynamics, if one is given a suitable cocycle representation $(\phi,\Iu): (A,\alpha)\to(\CM(B),\beta)$, then there appear to be two distinct but equally natural generalizations of the \cstar-algebraic object $\SD_\phi$, which we denote by $\SD_{(\phi,\Iu)}$ and $\SD^w_{(\phi,\Iu)}$, respectively.
These are in general non-isomorphic, and indeed we can detect the conceptual difference between them by describing their $K$-theory.
As the main result of the fifth section, it is shown as part of \autoref{thm:Paschke-2} that if $\FC$ is an eligible set with an absorbing element $(\phi,\Iu)\in\FC$, then $K_1(\SD^w_{(\phi,\Iu)})$ is naturally isomorphic to $KK^G(\FC;\alpha,\beta)$.
The group $K_1(\SD_{(\phi,\Iu)})$, on the other hand, turns out to be isomorphic to the direct sum $KK^G(\FC;\alpha,\beta)\oplus H_\beta$, where $H_\beta$ is the group of homotopy classes of cocycle pairs for $\beta$, which contains no information related to the action $\alpha$ but is often nontrivial when $G\neq\{1\}$.

As an outlook for future applications of the main results, Hua--White's recent article \cite{HuaWhite26} serves as an excellent template.
Therein, the authors prove a special case of the $KK$-uniqueness theorem, tailored to their intended application to classify certain nuclear embeddings of exact \cstar-algebras into ultrapowers of matrix algebras.
Their setup did not match the assumptions required for the classification machinery in \cite{CGSTW23}, but their work introduces a streamlined framework to extend the scope of such classification results if one has access to the general $KK$-uniqueness theorem.
I thus expect that \autoref{theorem:A} and its variants will be used in future works to obtain a more general classification of nuclear embeddings of exact \cstar-algebras into ultraproducts of \cstar-algebras that need not be Jiang--Su stable.
This expectation appears to be supported by the discussion in \cite[1.3.6]{CGSTW23}.
Similar consequences can be predicted within the classification theory of \cstar-dynamics.

%%%%%%%%%%%%%%%%%%%%%%%%%%%%%%%%%%%%%%%%%%%%%%%%%%%%%

\section{Preliminaries}

\begin{nota} \label{basic-notation}
Throughout, $G$ will denote a second-countable, locally compact group unless specified otherwise.
Normal capital letters like $A, B, C$ denote \cstar-algebras that are frequently assumed to be separable.
The multiplier algebra of $A$ is denoted as $\CM(A)$.
%, whereas $A^\dagger$ denotes the proper unitization of $A$, i.e., one adds a new unit even if $A$ was unital.
%We sometimes denote the closed unital ball of $A$ by $A_{\leq 1}$.
We write $\CU(\eins+A)$ for the set of all unitaries in the proper unitization $A^\dagger$ whose scalar part is $\eins$, which can be canonically identified with the unitary group of $A$ if it was already unital.
If we are given a (closed two-sided) ideal $J\triangleleft A$ and two elements $a,b\in A$, then the expression $a\equiv_J b$ is short-hand for $a-b\in J$.
The symbol $\CK$ denotes the \cstar-algebra of compact operators on a separable infinite-dimensional Hilbert space $\CH$, and we write $\CK(\CH)$ when a specific description of $\CH$ is relevant in a given context.
A \cstar-algebra $A$ is called \emph{stable} when $A\cong A\otimes\CK$.
Greek letters such as $\alpha,\beta,\gamma$ are used for point-norm continuous maps $G\to\Aut(A)$, most often group actions, which uniquely induces maps into $\Aut(\CM(A))$ as well.
Depending on the situation, we may denote $\id_A$ either for the identity map on $A$ or the trivial $G$-action on $A$.
We denote by $A^\alpha$ or $\CM(A)^\alpha$ the \cstar-subalgebra of fixed points (in $A$ or $\CM(A)$) with respect to an action $\alpha$.
Normal alphabetical letters such as $u,v,U,V$ are often used for unitary elements in some \cstar-algebra $A$.
If either $u\in\CU(\CM(A))$ or $u\in\CU(\eins+A)$, we denote by $\ad(u)$ the induced inner automorphism of $A$ given by $a\mapsto uau^*$.
Double-struck letters such as $\Iu, \Iv, \IU, \IV$ are used for strictly continuous maps $G\to\CU(\CM(A))$, most often \emph{cocycles} with respect to an action $\alpha: G\curvearrowright A$, i.e., one has the cocycle identity $\Iu_{gh}=\Iu_g\alpha_g(\Iu_h)$ for all $g,h\in G$.
Under this assumption, one obtains a new (cocycle perturbed) action $\alpha^\Iu: G\curvearrowright A$ via $\alpha^\Iu_g=\ad(\Iu_g)\circ\alpha_g$.

Although we will introduce equivariant $KK$-theory in this section via the Cuntz--Thomsen picture \cite{Cuntz83, Cuntz84, Higson87, Thomsen98}, we will implicitly assume that the reader has some existing passing familiarity with it; see for example \cite{BlaKK}.
We will assume that the reader is at least fluent in the language of $K$-theory for \cstar-algebras.
For the readers who are mainly interested in the results pertaining to $KK$-theory for \cstar-algebras (i.e., nonequivariant $KK$-theory), I wish to point out that one can choose to read the entire article by substituting $G=\{1\}$ and simply skip all the steps that involve the group actions.
\end{nota}

We will make use of the following fundamental result due to Kasparov concerning quasicentral approximate units coming from ideals invariant under a group action.

\begin{lemma}[see {\cite[Lemma 1.4]{Kasparov88}} and its proof] \label{lem:Kasparov}
Let $\beta: G\curvearrowright B$ be an action on a $\sigma$-unital \cstar-algebra and let $e_n\in B$ be any countable increasing approximate unit.
Then for any separable \cstar-subalgebra $D\subseteq\CM(B)$, there exists a countable increasing approximate unit of positive contractions $h_n\in B$ belonging to the convex hull of $\{e_n\mid n\geq 1\}$
satisfying 
\[
\lim_{n\to\infty} \|[h_n,d]\|=0 \quad\text{and}\quad \lim_{n\to\infty} \max_{g\in K} \| h_n-\beta_g(h_n) \| = 0
\] 
for all $d\in D$ and compact sets $K\subseteq G$.
If one has $e_{n+1}e_n=e_n$ for all $n\geq 1$, then one can arrange $h_{n+1}h_n=h_n$ for all $n\geq 1$.
\end{lemma}

\subsection{Some cocycle representation theory}

\begin{defi}[see {\cite[Section 1]{Szabo21cc}}] \label{def:cocycle-morphism}
Let $\alpha: G\curvearrowright A$ and $\beta: G\curvearrowright B$ be two actions on \cstar-algebras.
A \emph{cocycle representation} $(\phi,\Iu): (A,\alpha)\to (\CM(B),\beta)$ consists of a $*$-homomorphism $\phi: A\to\CM(B)$ and a strictly continuous $\beta$-cocycle $\Iu: G\to\CU(\CM(B))$ satisfying $\ad(\Iu_g)\circ\beta_g\circ\phi=\phi\circ\alpha_g$ for all $g\in G$.
For any unitary $W\in\CU(\CM(B))$, we can obtain a new cocycle representation defined by conjugation as 
\[
\ad(W)\circ(\phi,\Iu):=(\ad(W)\circ\phi, \{ W\Iu_g\beta_g(W)^*\}_{g\in G} ).
\]
If a given cocycle representation can be obtained from another via this procedure, then they are called \emph{unitarily equivalent}.
If there exists a cocycle representation $(\phi,\Iu)$ such that $\phi$ is an isomorphism from $A$ to $B$, then $\alpha$ and $\beta$ are called \emph{cocycle conjugate}.
\end{defi}

\begin{nota} \label{nota:strongly-stable}
We say that an action $\beta: G\curvearrowright B$ is \emph{strongly stable} if $(B,\beta)$ is  conjugate (i.e., equivariantly isomorphic) to $(B\otimes\CK,\beta\otimes\id_\CK)$.
\end{nota}

An action $\beta: G\curvearrowright B$ is strongly stable if and only if there is a sequence of isometries $r_n\in\CM(B)^\beta$ such that $\eins=\sum_{n=1}^\infty r_nr_n^*$ holds in the strict topology; see \cite[Remark 2.4]{GabeSzabo25}.
Furthermore, it was observed in \cite[Proposition 1.4]{GabeSzabo24} that every $G$-action on a stable \cstar-algebra is cocycle conjugate to a strongly stable action.

\begin{defi}[cf.\ {\cite[Lemma 3.4]{Thomsen98}}]
Let $\beta: G\curvearrowright B$ be an action on a \cstar-algebra.
Suppose that there exists a unital inclusion $\CO_2\subseteq\CM(B)^\beta$.
For two isometries $t_1,t_2\in\CM(B)^\beta$ with $t_1t_1^*+t_2t_2^*=\eins$, we may consider the $\beta$-equivariant $*$-homomorphism
\[
\CM(B)\oplus\CM(B)\to\CM(B),\quad b_1\oplus b_2\mapsto b_1\oplus_{t_1,t_2} b_2 := t_1b_1t_1^*+t_2b_2t_2^*.
\]
Up to unitary equivalence with a unitary in $\CM(B)^\beta$, this $*$-homomorphism does not depend on the choice of $t_1$ and $t_2$.\footnote{If $v_1, v_2\in\CM(B)^\beta$ are two other isometries with $v_1v_1^*+v_2v_2^*=\eins$, then the unitary equivalence between ``$\oplus_{t_1,t_2}$'' and ``$\oplus_{v_1,v_2}$'' is implemented by $w=t_1v_1^*+t_2v_2^*\in\CM(B)^\beta$.}
One refers to the element $b_1\oplus_{t_1,t_2} b_2$ as the \emph{Cuntz sum} of the two elements $b_1$ and $b_2$ (with respect to $t_1$ and $t_2$).
Now let $\alpha: G\curvearrowright A$ be another action on a \cstar-algebra, and $(\phi,\Iu), (\psi,\Iv): (A,\alpha)\to (\CM(B),\beta)$ two cocycle representations.
We likewise define the (pointwise) \emph{Cuntz sum}
\[
(\phi,\Iu)\oplus_{t_1,t_2} (\psi,\Iv) = (\phi\oplus_{t_1,t_2}\psi, \Iu\oplus_{t_1,t_2}\Iv): (A,\alpha)\to (\CM(B),\beta),
\]
which is easily seen to be another cocycle representation.
Since its unitary equivalence class does not depend on the choice of $t_1$ and $t_2$, we will often omit $t_1$ and $t_2$ from the notation if it is clear from context that a given statement is invariant under unitary equivalence.
\end{defi}

Similarly to how we formed Cuntz sums of two elements, we may also form countably infinite sums by a similar method if the underlying action is strongly stable.

\begin{convention} \label{standing-assumptions}
We shall, from this point and throughout the article, always assume that $\beta: G\curvearrowright B$ denotes a strongly stable action on a $\sigma$-unital \cstar-algebra and that $\alpha: G\curvearrowright A$ is an action on a separable \cstar-algebra, unless specified otherwise.
As we set out in \autoref{basic-notation}, $G$ is any second-countable locally compact group.
We shall only repeat these assumptions sporadically and when stating the main results.
\end{convention}

\begin{defi} \label{def:countable-sums}
%Suppose that $\beta$ is strongly stable.
Let $t_n\in \CM(B)^\beta$ be any sequence of isometries such that $\sum_{n=1}^\infty t_nt_n^* = \eins$ in the strict topology.
Then we have a $\beta$-equivariant $*$-homomorphism
\[
\ell^\infty(\IN, \CM(B)) \to \CM(B),\quad (b_n)_{n\geq 1} \mapsto \sum_{n=1}^\infty t_nb_nt_n^*,
\]
which does not depend on the choice of $t_n$ up to unitary equivalence with a unitary in $\CM(B)^\beta$.\footnote{Similarly as before, if $v_n\in\CM(B)^\beta$ is another sequence of isometries satisfying the same relation, then the unitary $w=\sum_{n=1}^\infty t_nv_n^*$ implements this equivalence.}
For any sequence of cocycle representations $(\phi^{(n)},\Iu^{(n)}): (A,\alpha)\to (\CM(B),\beta)$, we may hence define the countable sum
\[
(\Phi,\IU)=\bigoplus_{n=1}^\infty (\phi^{(n)},\Iu^{(n)}): (A,\alpha)\to (\CM(B),\beta)
\]
via the pointwise strict limits
\[
\Phi(a)=\sum_{n=1}^\infty t_n\phi^{(n)}(a)t_n^*,\quad \IU_g = \sum_{n=1}^\infty t_n\Iu^{(n)}_g t_n^*.
\]
Up to equivalence with a unitary in $\CM(B)^\beta$, this cocycle representation does not depend on the choice of $(t_n)_n$.
In particular, in the special case that $(\phi^{(n)},\Iu^{(n)})=(\phi,\Iu)$ for all $n$, we denote the resulting countable sum by $(\phi^\infty, \Iu^\infty)$ or $(\phi, \Iu)^\infty$ and call it the \emph{infinite repeat} of $(\phi,\Iu)$.
\end{defi}

The last sentence of the following result and its justification follows the same line of reasoning as \cite[Lemma 2.3]{Cuntz83} and \cite[Lemma 6.1]{Thomsen05}.

\begin{prop} \label{prop:infinite-repeat}
Let $(\phi,\Iu): (A,\alpha)\to (\CM(B),\beta)$ be a cocycle representation.
Then $(\phi,\Iu)$ is the infinite repeat of some cocycle representation if and only if there exists a sequence of isometries $r_n\in (\CM(B)\cap\phi(A)'\cap\{\Iu_g\}_{g\in G}')^{\beta}$ with $\sum_{n=1}^\infty r_n r_n^*=\eins$ in the strict topology.
If this is the case, then the \cstar-algebra $(\CM(B)\cap\phi(A)')^{\beta^\Iu}$ has vanishing $K$-theory.
\end{prop}
\begin{proof}
The ``if'' part is trivial: If $r_n$ is such a sequence of isometries, then forming the infinite repeat of $(\phi,\Iu)$ via this chosen sequence yields the (on the nose) equality $(\phi,\Iu)=(\phi,\Iu)^\infty$.

For the ``only if'' part, suppose that $(\phi,\Iu)$ is the infinite repeat of some cocycle representation $(\psi,\Iv)$.
Let $t_n\in\CM(B)^\beta$ be a sequence of isometries with $\sum_{n=1}^\infty t_n t_n^*=\eins$ and such that
\[
\phi(a)=\sum_{n=1}^\infty t_n\psi(a)t_n^*,\quad a\in A,
\]
and
\[
\Iu_g=\sum_{n=1}^\infty t_n\Iv_g t_n^*,\quad g\in G.
\]
Choose some sequence of pairwise disjoint infinite subsets $X_j\subset\IN$ with $\IN=\bigsqcup_{j=1}^\infty X_j$.
For each $j\geq 1$ we choose a bijection $\kappa_j: \IN\to X_j$ and consider the strict limit $r_j=\sum_{n=1}^\infty t_{\kappa_j(n)} t_n^*$.
We see all these isometries commute with the range of both $\phi$ and $\Iu$ pointwise and thus belong to the \cstar-algebra $(\CM(B)\cap\phi(A)'\cap\{\Iu_g\}_{g\in G}')^{\beta}$.
Furthermore the sequence satisfies $\sum_{n=1}^\infty r_n r_n^*=\eins$ in the strict topology.

For the last part of the statement, we keep a sequence $r_n$ chosen as in the statement.
Consider the strict limit $r_\infty=\sum_{n=1}^\infty r_{n+1}r_n^*$, which is an isometry in $(\CM(B)\cap\phi(A)')^{\beta^\Iu}$ and fits into the equation $r_1 r_1^* + r_\infty r_\infty^*=\eins$.
For any given element $x\in (\CM(B)\cap\phi(A)')^{\beta^\Iu}$, we can form its infinite repeat $x^\infty$ within $(\CM(B)\cap\phi(A)')^{\beta^\Iu}$ via the sequence $(r_n)_n$ and observe the equation
\[
x^\infty = \sum_{n=1}^\infty r_n x r_n^* = r_1 x r_1^* + r_\infty\Big( \sum_{n=1}^\infty r_n x r_n^* \Big) r_\infty^* = x\oplus_{r_1,r_\infty} x^\infty.
\]
Due to the additive properties of $K$-theory under Cuntz addition, this implies that the identity map of $(\CM(B)\cap\phi(A)')^{\beta^\Iu}$ induces the zero morphism on its $K$-groups.
This means that the $K$-groups vanish.
\end{proof}

\begin{defi} \label{def:various-equivalences}
Let $(\phi,\Iu), (\psi,\Iv): (A,\alpha)\to (\CM(B),\beta)$ be two cocycle representations.
We write $(\phi,\Iu)\asymp (\psi,\Iv)$, if there exists a norm-continuous path $U: [0,\infty)\to\CU(\CM(B))$ such that
	\begin{enumerate}[leftmargin=5mm,label={$\bullet$}]
  	\item $\psi(a) = \dst\lim_{t\to\infty} U_t \phi(a) U_t^*$ for all $a\in A$;
  	\item $\psi(a) - U_t \phi(a) U_t^* \in B$ for all $a\in A$ and $t\geq 0$;
  	\item $\dst\lim_{t\to\infty} \max_{g\in K}\ \| \Iv_g - U_t \Iu_g \beta_g(U_t)^* \| =0$ for all compact sets $K\subseteq G$;
  	\item $\Iv_g - U_t \Iu_g \beta_g(U_t)^* \in B$ for all $t\geq 0$ and $g\in G$.
  \end{enumerate} 
If it is possible to choose $U$ to have its range in $\CU(\eins+B)$, then $(\phi,\Iu)$ and $(\psi,\Iv)$ are called \emph{properly asymptotically unitarily equivalent}.
If it is additionally possible to arrange for $U_0=\eins$, then $(\phi,\Iu)$ and $(\psi,\Iv)$ are called \emph{strongly asymptotically unitarily equivalent}.
\end{defi}

\begin{defi} \label{def:absorption}
%Assume that $\beta$ is strongly stable.
Let $(\phi,\Iu), (\psi,\Iv): (A,\alpha)\to (\CM(B),\beta)$ be two cocycle representations.
\begin{enumerate}[label=\textup{(\roman*)}, leftmargin=*]
\item We say that $(\phi,\Iu)$ \emph{absorbs} $(\psi,\Iv)$ if $(\phi\oplus\psi,\Iu\oplus\Iv)\asymp (\phi,\Iu)$.
\item We say that $(\phi,\Iu)$ \emph{weakly absorbs} $(\psi,\Iv)$ if there exists a unitary $W\in\CU(\CM(B))$ such that $\phi(a)-W(\phi\oplus\psi)(a)W^*\in B$ for all $a\in A$ and $\Iu_g-W(\Iu_g\oplus\Iv_g)\beta_g(W)^*\in B$ for all $g\in G$.
\end{enumerate}
\end{defi}

In line with the above condition, a given cocycle representation $(\phi,\Iu)$ is called \emph{absorbing} if it absorbs every cocycle representation $(A,\alpha)\to(\CM(B),\beta)$.
It was proved as part of \cite[Theorem 4.16]{GabeSzabo25} that such representations always exist as long as both $A$ and $B$ are separable.

\begin{defi}[{\cite[Definition 4.3]{GabeSzabo25}}] \label{def:wc}
Let $(\psi,\Iv): (A,\alpha) \to (\CM(B),\beta)$ be a cocycle representation, and let $\FC$ be a family of cocycle representations from $(A,\alpha)$ to $(\CM(B),\beta)$.
We say that $(\psi,\Iv)$ is \emph{weakly contained} in $\FC$, written $(\psi,\Iv)\wc \FC$,
if the following is true.

For all compact sets $K\subseteq G$, $\CF\subset A$, every $\eps> 0$, and every contraction $b\in B$ there exist $(\phi^{(1)},\Iu^{(1)}),\dots,(\phi^{(\ell)},\Iu^{(\ell)})\in\FC$ and a collection of elements
\[
\{c_{j,k} \mid j=1,\dots, N,\ k=1,\dots,\ell\}\subset B
\]
satisfying
\[
\max_{g\in K} \Big\| b^*\Iv_g\beta_g(b)-\sum_{k=1}^\ell \sum_{j=1}^{N} c_{j,k}^*\Iu^{(k)}_g\beta_g(c_{j,k}) \Big\| \leq \eps
\]
and
\[
\max_{a\in\CF} \Big\| b^*\psi(a)b - \sum_{k=1}^\ell \sum_{j=1}^N c_{j,k}^*\phi^{(k)}(a)c_{j,k} \Big\| \leq \eps.
\]
In particular, if $\FC$ consists of a single cocycle representation $(\phi,\Iu)$, then we write $(\psi,\Iv)\wc (\phi,\Iu)$ in place of $(\psi,\Iv)\wc\set{(\phi,\Iu)}$, and in this case one can choose $\ell=1$ above.
\end{defi}

\begin{defi}[{\cite[Definition 4.5]{GabeSzabo25}}] \label{def:contained-at-infty}
Let $(\phi,\Iu), (\psi,\Iv): (A,\alpha) \to (\CM(B),\beta)$ be two cocycle representations.
We say that $(\psi,\Iv)$ is \emph{contained in $(\phi,\Iu)$ at infinity}, if the following is true.
For all compact sets $K\subseteq G$, $\CF\subset A$, every $\eps> 0$, and every contraction $b\in B$ there exists an element $x\in B$ satisfying
\[
\max_{g\in K} \| b^*\Iv_g\beta_g(b)-x^*\Iu_g\beta_g(x) \| \leq \eps,\ \max_{a\in\CF} \| b^*\psi(a)b - x^*\phi(a)x \| \leq \eps,\ \| x^*b \| \leq \eps.
\]
\end{defi}

\begin{prop}[{\cite[Lemma 4.11]{GabeSzabo25}}] \label{prop:almost-infinite-repeat}
Let $(\phi,\Iu), (\psi,\Iv): (A,\alpha)\to (\CM(B),\beta)$ be two cocycle representations.
The following are equivalent:
\begin{enumerate}[label=\textup{(\roman*)},leftmargin=*]
\item \label{lem:abs-inf-repeat:1}
$(\psi,\Iv)$ is contained in $(\phi,\Iu)$ at infinity;
\item \label{lem:abs-inf-repeat:2}
$(\phi,\Iu)$ absorbs $(\psi^\infty,\Iv^\infty)$.
\item \label{lem:abs-inf-repeat:3}
$(\phi,\Iu)$ weakly absorbs $(\psi^\infty,\Iv^\infty)$.
\end{enumerate}
In particular, $(\phi,\Iu)$ contains itself at infinity if and only if modulo $B$, it can be expressed as the unitary conjugate of an infinite repeat.
\end{prop}

%%%

\subsection{The Cuntz--Thomsen picture of equivariant $KK$-theory}

\begin{defi}[cf.\ {\cite[Section 3]{Thomsen98}}] \label{def:equi-Cuntz-pair}
Let $\alpha: G\curvearrowright A$ and $\beta: G\curvearrowright B$ be two actions on \cstar-algebras such that $A$ is separable and $B$ is $\sigma$-unital.
An \emph{(equivariant) $(\alpha,\beta)$-Cuntz pair} is a pair of cocycle representations
\[
(\phi,\Iu), (\psi,\Iv): (A,\alpha) \to (\CM(B\otimes\CK),\beta\otimes\id_\CK),
\]
such that the pointwise differences $\phi-\psi$ and $\Iu-\Iv$ take values in $B\otimes\CK$.\footnote{In Thomsen's article it was also assumed that the map $\Iu-\Iv$ is norm-continuous. This turns out to be redundant, see \cite[Proposition 6.9]{Szabo21cc}.}
A pair of the specific form $\big( (\phi,\Iu), (\phi,\Iu) \big)$ is called \emph{degenerate}.
Whenever $\beta$ is assumed to be strongly stable (this is our default assumption), we also allow $(B,\beta)$ in place of $(B\otimes\CK,\beta\otimes\id_\CK)$ appearing in the definition of an equivariant $(\alpha,\beta)$-Cuntz pair.
\end{defi}

\begin{nota}
Given a \cstar-algebra $B$, we denote $B[0,1]=\CC[0,1]\otimes B$.
If one has an action $\beta: G\curvearrowright B$, we consider the induced $G$-action on $B[0,1]$ given by $\beta[0,1]=\id_{\CC[0,1]}\otimes\beta$.
\end{nota}

\begin{defi}[see {\cite[Section 3]{Thomsen98} and \cite[Section 2]{GabeSzabo25}}] \label{def:KKG-Thomsen}
Let $A$ be a separable \cstar-algebra and $B$ a $\sigma$-unital \cstar-algebra.
For two actions $\alpha: G\curvearrowright A$ and $\beta: G\curvearrowright B$,
let $\IE^G(\alpha,\beta)$ denote the set of all $(\alpha,\beta)$-Cuntz pairs.

Two elements $\big( (\phi^0,\Iu^0), (\psi^0,\Iv^0) \big)$ and $\big( (\phi^1,\Iu^1), (\psi^1,\Iv^1) \big)$ in $\IE^G(\alpha,\beta)$ are called \emph{homotopic}, abbreviated $\big( (\phi^0,\Iu^0), (\psi^0,\Iv^0) \big)\sim_h \big( (\phi^1,\Iu^1), (\psi^1,\Iv^1) \big)$, if there exists an $(\alpha,\beta[0,1])$-Cuntz pair that restricts to $\big( (\phi^0,\Iu^0), (\psi^0,\Iv^0) \big)$ upon evaluation at $0\in [0,1]$, and restricts to $\big( (\phi^1,\Iu^1), (\psi^1,\Iv^1) \big)$ upon evaluation at $1\in [0,1]$.
An $(\alpha,\beta)$-Cuntz pair of the form $\big( (\phi,\Iu), (\psi,\Iv) \big)$ with $\phi=\psi=0$ is called a \emph{cocycle pair} and is denoted by $(\Iu,\Iv)$ with slight abuse of notation.
We define $\IE_0^G(\alpha,\beta)$ as the set of all \emph{anchored} $(\alpha,\beta)$-Cuntz pairs, i.e., those $\big( (\phi,\Iu), (\psi,\Iv) \big)\in\IE^G(\alpha,\beta)$ such that $(\Iu,\Iv)\sim_h (\eins,\eins)$.

For any unital inclusion $\CO_2\subseteq\CM(B\otimes\CK)^{\beta\otimes\id_\CK}$ with generating isometries $t_1,t_2$, one can perform the Cuntz addition for two $(\alpha,\beta)$-Cuntz pairs as
\[
\begin{array}{cl}
\multicolumn{2}{l}{
\big( (\phi^0,\Iu^0), (\psi^0,\Iv^0) \big) \oplus_{t_1,t_2} \big( (\phi^1,\Iu^1), (\psi^1,\Iv^1) \big) } \\
=& \big( (\phi^0,\Iu^0)\oplus_{t_1,t_2}(\phi^1,\Iu^1), (\psi^0,\Iv^0)\oplus_{t_1,t_2} (\psi^1,\Iv^1) \big).
\end{array}
\]
This Cuntz pair is independent of the choice of $t_1,t_2$ up to homotopy; see \cite[Lemma 3.4]{Thomsen98}.
\end{defi}

\begin{rem}[see {\cite[Proposition 2.12]{GabeSzabo25}}] \label{def:KKG}
The quotient $\IE^G(\alpha,\beta)/{\sim_h}$ becomes an abelian group with Cuntz addition.
The homotopy classes of cocycle pairs form a canonical subgroup $H_\beta$ independent of $\alpha$.
It was proved by Thomsen in \cite[Theorem 3.5]{Thomsen98} that the group quotient of $\IE^G(\alpha,\beta)/{\sim_h}$ modulo $H_\beta$ is naturally isomorphic to $KK^G(\alpha,\beta)$ if one defines the latter via Kasparov's original approach \cite{Kasparov88}.
For an $(\alpha,\beta)$-Cuntz pair consisting of $(\phi,\Iu)$ and $(\psi,\Iv)$, we denote its associated homotopy class by $[(\phi,\Iu),(\psi,\Iv)]$.
Under this identification, one has that the inclusion map $\IE^G_0(\alpha,\beta)\subseteq\IE^G(\alpha,\beta)$ also induces a natural isomorphism of abelian groups $\IE^G_0(\alpha,\beta)/{\sim_h}\cong KK^G(\alpha,\beta)$.
In other words, $KK^G(\alpha,\beta)$ may be defined as the abelian group of homotopy classes of anchored $(\alpha,\beta)$-Cuntz pairs.
\end{rem}

%%%%%%%%%%

\section{$K_1$-injectivity of certain auxiliary \cstar-algebras}

The following notation was already used in \cite{GabeSzabo25} and is inspired by notation appearing in the work of Dadarlat--Eilers in \cite[Section 3]{DadarlatEilers01}.
We note that for \cstar-dynamics, a very similar type of auxiliary \cstar-algebras was considered by Thomsen in \cite[Section 6]{Thomsen05} and its conceptual role was similar to how we will use the objects below in this article.

\begin{nota} \label{nota:D-phi}
Let $\alpha: G\curvearrowright A$ and $\beta: G\curvearrowright B$ be actions on \cstar-algebras.
Let $(\phi,\Iu): (A,\alpha)\to (\CM(B),\beta)$ be a cocycle representation.
We then consider the two \cstar-algebras
\[
\SD_{(\phi,\Iu)}=\set{ x\in\CM(B) \mid [x,\phi(A)]\subseteq B, \set{x-\beta^\Iu_g(x)}_{g\in G}\subseteq B}
\]
and
\[
\SD^w_{(\phi,\Iu)}=\set{ x\in\CM(B) \mid [x,\phi(A)]\subseteq B, \set{ \phi(a)(\beta^\Iu_g(x)-x)}_{a\in A, g\in G}\subseteq B }.
\]
Note that these two objects coincide for $G=\{1\}$.
Indeed, if $\phi: A\to\CM(B)$ is any $*$-homomorphism, then it can be viewed as a cocycle representation for actions of the trivial group, and we simply write 
\[
\SD_\phi=\set{ x\in\CM(B) \mid [x,\phi(A)]\subseteq B}.
\]
\end{nota}

For the rest of this section, we continue to follow \autoref{standing-assumptions}.

\begin{rem} \label{rem:D-infinite-repeat}
Let $(\phi,\Iu),(\psi,\Iv): (A,\alpha)\to(\CM(B),\beta)$ be two cocycle representations.
Suppose there exists $U\in\CU(\CM(B))$ such that $(\phi,\Iu)$ agrees with $\ad(U)\circ(\psi,\Iv)$ modulo $B$.
Based on the definitions, we can directly conclude that $\ad(U)\in\Aut(\CM(B))$ restricts to isomorphisms 
\[
\SD^w_{(\psi,\Iv)}\cong \SD^w_{(\phi,\Iu)} \quad\text{and}\quad \SD_{(\psi,\Iv)}\cong \SD_{(\phi,\Iu)}.
\]
If we assume that $(\phi,\Iu)$ contains itself at infinity and apply \autoref{prop:almost-infinite-repeat}, one obtains isomorphisms $\SD^w_{(\phi,\Iu)}\cong\SD^w_{(\phi,\Iu)^\infty}$ and $\SD_{(\phi,\Iu)}\cong\SD_{(\phi,\Iu)^\infty}$.
\end{rem}

For what follows, we recall that a unital \cstar-algebra $D$ is called \emph{properly infinite} if it contains two isometries $s,t\in D$ with $s^*t=0$.

\begin{nota}[see {\cite[Subsection 4.2]{KirchbergC}}]
Let $D$ be a properly infinite unital \cstar-algebra.
We say that a projection $p\in D$ is a \emph{splitting projection}, if there exist isometries $s,t\in D$ such that $ss^*\leq p$ and $tt^*\leq\eins-p$.
We say that $D$ is \emph{in Cuntz standard form}, if there exists a unital inclusion $\CO_2\subseteq D$.
\end{nota}

\begin{rem} \label{rem:Cuntz}
It follows from Cuntz's fundamental article \cite{Cuntz81} that the canonical map $p\mapsto [p]_0$ induces a one-to-one correspondence between unitary equivalence classes of splitting projections in $D$ and the group $K_0(D)$.
Furthermore, if $D$ is in Cuntz standard form and $p\in D$ is a splitting projection, then $[p]_0=0$ if and only if $p\sim_\MvN\eins-p\sim_\MvN\sim\eins$.
Note that in \cite{CuntzHigson87}, projections with the latter property were referred to as ``proper'' projections.
We shall not use this terminology because it may be considered too ambiguous from today's point of view, but wish to point this out to the reader as we will reference said article below.
\end{rem}

\begin{lemma} \label{lem:Cuntz-Higson}
Let $D$ be a properly infinite unital \cstar-algebra in Cuntz standard form.
Let $p,q\in D$ be two splitting projections with vanishing $K_0$-class.
Suppose that there exist isometries $s_1,s_2\in D$ such that $s_1s_1^*\leq p$, $s_2 s_2^*\leq q$ and $\|s_1^*s_2\|<1$.
Then $p$ and $q$ are homotopic.
\end{lemma}
\begin{proof}
As both $p$ and $q$ are splitting projections, it immediately follows for $j=1,2$ that the range projections $s_js_j^*$ are also splitting projections with vanishing $K_0$-class.
Using Cuntz--Higson's \cite[Lemma 1]{CuntzHigson87} repeatedly, it follows that we have a chain of homotopies
\[
p \sim_h \eins-p \sim_h s_1s_1^* \sim_h s_2s_2^* \sim_h \eins-q \sim_h q. 
\]
This finishes the argument.
\end{proof}

Recall that a unital \cstar-algebra $A$ is called \emph{$K_1$-injective} if the canonical map $\CU(A)/\CU_0(A)\to K_1(A)$ is injective.
The following serves as our main criterion for $K_1$-injectivity of properly infinite \cstar-algebras.

\begin{lemma} \label{lem:K1-inj-criterion}
Let $D$ be a properly infinite unital \cstar-algebra.
The following are equivalent:
\begin{enumerate}[label=\textup{(\roman*)},leftmargin=*]
	\item $D$ is $K_1$-injective. \label{lem:K1-inj-criterion:1}
	\item For every pair of splitting projections $p,q\in D$, one has $p\sim_{\MvN} q$ if and only if $p\sim_h q$. \label{lem:K1-inj-criterion:2}
	\item There exists some splitting projection $p\in D$ such that for every splitting projection $q\in D$, one has $p\sim_{\MvN} q$ if and only if $p\sim_h q$. \label{lem:K1-inj-criterion:3}
	\end{enumerate}
\end{lemma}
\begin{proof}
The implication \ref{lem:K1-inj-criterion:2}$\Rightarrow$\ref{lem:K1-inj-criterion:3} is tautological.
The equivalence \ref{lem:K1-inj-criterion:1}$\Leftrightarrow$\ref{lem:K1-inj-criterion:2} is part of \cite[Proposition 5.1]{BlanchardRohdeRordam08}, but a careful reading of its proof reveals that the argument given there actually proves \ref{lem:K1-inj-criterion:3}$\Rightarrow$\ref{lem:K1-inj-criterion:1}.
\end{proof}

The statement below is the consequence of a standard functional calculus argument.
We shall omit its easy proof.

\begin{prop} \label{prop:universal-constants}
For every $\eps>0$, there exists a constant $\delta>0$ with the following property.
Let $A$ be any \cstar-algebra and $a\in A$ a positive contraction.
\begin{enumerate}[leftmargin=5mm,label={$\bullet$}]
\item If $x\in A$ is any contraction with $\|[x,a]\|<\delta$, then $\|[x,\sqrt{a}]\|<\eps$.
\item If $\alpha\in\Aut(A)$ is any automorphism with $\|a-\alpha(a)\|<\delta$, then $\|\sqrt{a}-\alpha(\sqrt{a})\|<\eps$.
\end{enumerate}
\end{prop}

The following is the main technical result of this section.
A special case of it, stated after it, is the main ingredient to obtain \autoref{theorem:A}, and its general form enables many of the results in the rest of the article.

\begin{theorem} \label{thm:K1-injectivity}
Let $\alpha: G\curvearrowright A$ be an action on a separable \cstar-algebra and $\beta: G\curvearrowright B$ a strongly stable action on a $\sigma$-unital \cstar-algebra.
Let $(\phi,\Iu): (A,\alpha)\to(\CM(B),\beta)$ be a cocycle representation that contains itself at infinity.
Let $\CE$ be a unital \cstar-algebra for which one of the following holds:
	\begin{enumerate}[label=\textup{(\arabic*)},leftmargin=*]
	\item There exists an ideal $\CJ\triangleleft \SD_{(\phi,\Iu)}^w$ such that $\CE\cong \SD_{(\phi,\Iu)}^w/\CJ$.
	\item There exists an ideal $\CJ\triangleleft \SD_{(\phi,\Iu)}$ such that $\CE\cong \SD_{(\phi,\Iu)}/\CJ$.
	\end{enumerate}
Then $\CE$ is $K_1$-injective.
\end{theorem}
\begin{proof}
By \autoref{rem:D-infinite-repeat}, we may assume without loss of generality that $(\phi,\Iu)$ is an infinite repeat.
By \autoref{prop:infinite-repeat}, we may assume that there exists a sequence of isometries $r_n\in (\CM(B)\cap\phi(A)')^{\beta^\Iu}$ with $\eins=\sum_{n=1}^\infty r_n r_n^*$.
Depending on which case is proved, we write either $\SD=\SD_{(\phi,\Iu)}$ or $\SD=\SD_{(\phi,\Iu)}^w$, so that $\CE=\SD/\CJ$ in both cases.
We assume that $\CE$ is non-zero, else there is nothing to prove.
We write $\pi_\CJ: \SD\to\CE$ for the quotient map.

We consider $p=\pi_\CJ(P)$ for $P=\sum_{n=1}^\infty r_{2n}r_{2n}^*\in(\CM(B)\cap\phi(A)')^{\beta^\Iu}\subseteq\SD$, which defines a splitting projection in $\CE$ with vanishing $K_0$-class (as $P$ is already such an element in $\SD$).
Let $q\in\CE$ be an arbitrary splitting projection with vanishing $K_0$-class.
By \autoref{lem:K1-inj-criterion}, the claim amounts to showing that $p$ and $q$ are homotopic.
By \autoref{rem:Cuntz}, $q$ is uniquely determined modulo unitary equivalence, thus there exists a unitary $w\in\CE$ such that $w(\eins-p)w^*=q$.
Let $W\in\SD$ be a contraction with $w=\pi_\CJ(W)$.
Set $Q=W(\eins-P)W^* = W\big( \sum_{n=1}^\infty  r_{2n-1} r_{2n-1}^*\big) W^*$, which then satisfies $q=\pi_\CJ(Q)$.

We apply \autoref{lem:Kasparov} and choose an increasing approximate unit $(e_n)_{n\geq 1}$ in $B$ with $e_{n+1}e_n=e_n$ that is $\phi(A)$-quasicentral and approximately $\beta^\Iu$-invariant.
Let $K_n\subseteq G$ be an increasing sequence of compact sets with union $G$ and $F_n\subseteq A$ an increasing sequence of finite subsets of contractions whose union is dense in the unit ball of $A$.
For every $n\geq 1$, we apply \autoref{prop:universal-constants} to find a universal constant $\delta_n>0$ that satisfies the conclusion therein for $2^{-(n+2)}$ in place of $\eps$.
We assume $\delta_n\leq 2^{-n}$ for each $n\geq 1$ and that the sequence $(\delta_n)_n$ is decreasing.
We begin by choosing an even number $m_1\geq 1$ with 
\[
\max_{g\in K_1} \|e_n-\beta_g^\Iu(e_n)\| + \max_{a\in F_1} \|[e_n,\phi(a)]\|< \delta_1/2\ , \quad n\geq m_1.
\]
Set $f_1=r_{m_1} e_{m_1} r_{m_1}^*$.
As $r_nr_n^*\to 0$ strictly as $n\to\infty$, we may find an odd number $n_1> m_1$ such that $\|e_{m_1}r_{m_1}^*Wr_{n_1}\|<\frac14$ and set $h_1=Wr_{n_1}e_{n_1}r_{n_1}^*W^*$.
Next, choose an even number $m_2> n_1$ with
\[
\max_{g\in K_2} \|e_n-\beta_g^\Iu(e_n)\| + \max_{a\in F_2} \|[e_n,\phi(a)]\| < \delta_2/2 ,\quad n\geq m_2,
\]
and
\[
\|r_n^*Wr_{n_1}e_{n_1}\|<\frac18,\quad n\geq m_2.
\]
Set $f_2=r_{m_2}e_{m_2}r_{m_2}^*$.
Then find an odd number $n_2> m_2$ such that 
\[
\|e_{m_1}r_{m_1}^*Wr_{n_2}\|<\frac{1}{8},\quad \|e_{m_2}r_{m_2}^*Wr_{n_2}\|<\frac{1}{16},
\] 
and set $h_2=Wr_{n_2}e_{n_2}r_{n_2}^*W^*$.

By repeating these steps inductively, we end up with increasing sequences of (alternatingly even and odd) natural numbers $m_1< n_1< m_2< n_2<\dots$ such that
\begin{equation} \label{eq:properties-en}
\sup_{n\geq m_k}\Big( \max_{g\in K_k} \|e_n-\beta_g^\Iu(e_n)\| + \max_{a\in F_k} \|[e_n,\phi(a)]\| \Big) < \delta_k/2,\quad k\geq 1, 
\end{equation}
and the positive contractions
\[
f_k=r_{m_k} e_{m_k} r_{m_k}^*,\quad h_k=Wr_{n_k} e_{n_k} r_{n_k}^*W^*
\]
satisfy the inequality $\|f_k h_j\|<2^{-(j+k)}$ for all $j,k\geq 1$.

Consider the strict limits $\FF=\sum_{k=1}^\infty f_k$ and $\FH=\sum_{k=1}^\infty h_k$.
Clearly $\FF\leq P$ and $\FH\leq Q$ and the above inequality implies 
\begin{equation} \label{eq:product-FH}
\|\FF\cdot\FH\|\leq\sum_{j,k=1}^\infty \|f_k h_j\|<\sum_{j,k=1}^\infty 2^{-(j+k)}=1.
\end{equation}
%%%
\begin{comment}
As a consequence of \eqref{eq:properties-en}, we observe for all $a\in A$
\[
[\phi(a),\FG] = \sum_{k=1}^\infty r_{m_k}[\phi(a),e_{m_k}]r_{m_k}^* \in B
\]
and, since $[W,\phi(A)]\subseteq B$, also
\[
\begin{array}{ccl}
[\phi(a),\FH] &=& \dst \phi(a)W\Big(\sum_{k=1}^\infty r_{n_k} e_{n_k} r_{n_k}^*\Big) W^*-W\Big(\sum_{k=1}^\infty r_{n_k} e_{n_k} r_{n_k}^*\Big) W^*\phi(a) \\
&\equiv& \dst W\Big(\sum_{k=1}^\infty r_{n_k}[\phi(a),e_{n_k}]r_{n_k}^* \Big)W^* \in B.
\end{array}
\]
Next we clearly have $\FG \in \CM^{\beta^\Iu}(B)$ and hence $\FG\in D_{(\phi,\Iu)}\subseteq\CE'$.
If $\CE'=D_{(\phi,\Iu)}\ni W$, then we observe for all $g\in G$ that
\[
\begin{array}{ccl}
\FH-\beta_g^\Iu(\FH) &=& \dst \sum_{k=1}^\infty Wr_{n_k} e_{n_k} r_{n_k}^* W - \beta_g^\Iu( Wr_{n_k} e_{n_k} r_{n_k}^* W ) \\
&\equiv& \dst W\Big( \sum_{k=1}^\infty r_{n_k} (e_{n_k}-\beta_g^\Iu(e_{n_k}) ) r_{n_k}^* \Big)W^* \in B. 
\end{array}
\]
If however $\CE'=D^w_{(\phi,\Iu)}\ni W$, then we observe for all $a\in A$ and $g\in G$ that
\[
\begin{array}{cl}
\multicolumn{2}{l}{ \phi(a)(\FH-\beta_g^\Iu(\FH))\phi(a)^* }\\
=& \dst \phi(a)\Big( \sum_{k=1}^\infty Wr_{n_k} e_{n_k} r_{n_k}^* W^* - \beta_g^\Iu( Wr_{n_k} e_{n_k} r_{n_k}^* W^* ) \Big)\phi(a) \\
\equiv& \dst \phi(a)W\Big( \sum_{k=1}^\infty r_{n_k} (e_{n_k}-\beta_g^\Iu(e_{n_k}) ) r_{n_k}^* \Big)W^*\phi(a)^* \in B. 
\end{array}
\]
With \autoref{lemma:weak-fixed-point-property}, we conclude $\FH\in\CE$.
\end{comment}
%%%

For notational convenience, set $e_0:=0=:e_{-1}$ and $m_0=n_0=0=m_{-1}=n_{-1}$.
For $k\geq 1$, define $a_k=(e_{m_{k-1}}-e_{m_{k-2}})^{1/2}$ and $b_k=(e_{n_{k-1}}-e_{n_{k-2}})^{1/2}$.
As $(e_n)_n$ was an approximate unit, note that
\[
\sum_{k=1}^\infty a_k^2 = \lim_{\ell\to\infty} e_{m_{\ell-1}}= \eins = \lim_{\ell\to\infty} e_{n_{\ell-1}}= \sum_{k=1}^\infty b_k^2.
\]
By condition \eqref{eq:properties-en}, we have for all $k\geq 3$ that
\[
\max_{g\in K_{k-2}} \|(e_{m_{k-1}}-e_{m_{k-2}})-\beta_g^\Iu(e_{m_{k-1}}-e_{m_{k-2}})\| \leq \delta_{k-2}
\]
and
\[
\max_{a\in F_{k-2}} \|[(e_{m_{k-1}}-e_{m_{k-2}}),\phi(a)]\|\leq \delta_{k-2}.
\]
By our choice of the constants $\delta_k$, it follows for all $k\geq 3$ that
\begin{equation} \label{eq:properties-ak}
\max_{g\in K_{k-2}} \|a_k-\beta_g^\Iu(a_k)\| \leq 2^{-k},\quad \max_{a\in F_{k-2}} \|[a_k,\phi(a)]\|\leq 2^{-k}.
\end{equation}
Analogously, we get for all $k\geq 3$ that
\begin{equation} \label{eq:properties-bk}
\max_{g\in K_{k-2}} \|b_k-\beta_g^\Iu(b_k)\| \leq 2^{-k},\quad \max_{a\in F_{k-2}} \|[b_k,\phi(a)]\|\leq 2^{-k}.
\end{equation}
Consider the strict limit $R_1=\sum_{k=1}^\infty r_{m_k} a_k$.
It is an isometry since
\[
R_1^*R_1=\sum_{k,j=1}^\infty a_k r_{m_k}^* r_{m_j} a_j = \sum_{k=1}^\infty a_k^2 = \eins.
\]
We also have
\[
\FF R_1 = \sum_{k,j=1}^\infty r_{m_j} e_{m_j} r_{m_j}^* r_{m_k} a_k = \sum_{k=1}^\infty r_{m_k} \underbrace{e_{m_k} a_k}_{=a_k} = R_1.
\]
Condition \eqref{eq:properties-ak} implies for all $a\in\bigcup_n F_n$ that
\[
[R_1,\phi(a)]=\sum_{k=1}^\infty r_{m_k} [a_k,\phi(a)] \in B
\]
and for all $g\in G$ that
\[
R_1-\beta^\Iu_g(R_1)=\sum_{k=1}^\infty r_{m_k} \big( a_k - \beta^\Iu_g(a_k) \big) \in B,
\]
and hence $R_1\in \SD_{(\phi,\Iu)}\subseteq\SD$.

Likewise we can consider the strict limit $R_2'=\sum_{k=1}^\infty r_{n_k} b_k$ and $R_2=WR_2'$.
As above, condition \eqref{eq:properties-bk} implies $R_2'\in \SD_{(\phi,\Iu)}\subseteq\SD$ and thus $R_2\in\SD$.
Exactly as above it follows that $R_2'$ is an isometry and hence
\[
R_2^*R_2-\eins={R_2'}^{*}(W^*W-\eins)R_2' \in \CJ. 
\]
Here we used that $W$ was chosen as a lift of the unitary $w\in\CE$.
Furthermore,
\[
\begin{array}{cll}
\FH R_2 &=& W\Big(\sum_{k=1}^\infty r_{n_k} e_{n_k} r_{n_k}^*\Big)W^*W  R_2' \\
&\equiv_\CJ& W\Big(\sum_{k=1}^\infty r_{n_k} e_{n_k} r_{n_k}^*\Big)R_2' \\
&=& W\sum_{k,j=1}^\infty r_{n_j} e_{n_j} r_{n_j}^* r_{n_k} b_k \\
&=& W\sum_{k=1}^\infty r_{n_k} \underbrace{e_{n_k} b_k}_{=b_k} \\
&=& WR_2' \ = \ R_2.
\end{array}
\]
To summarize, if we set $s_j=\pi_\CJ(R_j)$ for $j=1,2$, then we have found two isometries $s_1,s_2\in\CE$ with 
\[
s_1^*s_1^* \leq \pi_\CJ(\FF) \leq \pi_\CJ(P)=p,\quad s_2s_2^*\leq\pi_\CJ(\FH)\leq \pi_\CJ(Q)= q,
\]
and
\[
\|s_1^*s_2\|=\|\pi_\CJ(R_1^*R_2)\|=\|\pi_\CJ(R_1^*\FF\FH R_2)\| \leq\|R_1^*\FF\FH R_2\|\leq\|\FF\FH\|\stackrel{\eqref{eq:product-FH}}{<}1.
\]
With \autoref{lem:Cuntz-Higson}, we may conclude that $p$ and $q$ are homotopic in $\CE$.
\end{proof}

We obtain the following as a consequence.
Note that the special case $A=0$ recovers \cite[Proposition 4.9]{GabeRuiz20}, the special case $B=\CK$ recovers \cite[Lemma 3(2)]{Paschke81}, and \cite[Theorem 2.5, 2.9]{LoreauxNg20}, \cite[Theorem 3.28]{LNS26} as well as \cite[Corollary 4.2]{HuaWhite26} are recovered as various other special cases.
The unitality condition written in brackets is an optional extra assumption one may add to the theorem in case $A$ is unital, so we obtain two versions of the statement below, with or without units.
In particular we also obtain \autoref{theorem:B}.

\begin{cor} \label{cor:Paschke-dual-K1-injective}
Let $A$ be a separable \cstar-algebra and $B$ a $\sigma$-unital stable \cstar-algebra.
Let $\phi: A\to\CM(B)$ be a (unitally) absorbing representation and let $\bar{\phi}: A\to\CQ(B)=\CM(B)/B$ be its induced homomorphism to the corona algebra.
Then the relative commutant $\CQ(B)\cap\bar{\phi}(A)'$ is $K_1$-injective.
The same conclusion holds if $\phi$ is instead assumed to be weakly nuclear and nuclearly (unitally) absorbing.
\end{cor}
\begin{proof}
We apply \autoref{thm:K1-injectivity} to the case $G=\{1\}$ with trivial actions on $A$ and $B$.
Since $\phi$ is either a (unitally) absorbing representation or is weakly nuclear and nuclearly (unitally) absorbing, it absorbs its infinite repeat $\phi^\infty$, so it contains itself at infinity.
If one keeps in mind that $B$ sits in $\SD_\phi$ as an ideal and compares definitions, one has $\CQ(B)\cap\bar{\phi}(A)'=\SD_\phi/B$.
So if we apply the conclusion of \autoref{thm:K1-injectivity} with $\CJ=B$, the claim follows.
\end{proof}

The above result implies a positive solution to the $KK$-uniqueness problem for \cstar-algebras posed in recent work of Carrión et al \cite[Question 5.17]{CGSTW23}; see also \cite[Theorem 2.6]{LNS26} or the discussion following \cite[Problem LXII]{SchafhauserTikuisisWhite25}.
We record the resulting theorem here, whose separable version implies \autoref{theorem:A}.
Despite outsourcing the justification of the theorem to these references for the time being, we note that the subsequent sections will lead to a self-contained proof of \autoref{theorem:C}, which is strictly more general than \ref{thm:original-KK-uniqueness:1} below.
Since the proof of \cite[Theorem 5.15]{CGSTW23} makes it is clear that the $KL$-uniqueness theorem (\ref{thm:original-KK-uniqueness:2} below) can be reduced to \ref{thm:original-KK-uniqueness:1}, we omit a separate proof and shall not introduce or discuss $KL$-theory in the rest of the article.

\begin{theorem} \label{thm:original-KK-uniqueness}
Let $A$ be a separable \cstar-algebra and $B$ a $\sigma$-unital stable \cstar-algebra.
Let $\phi,\psi: A\to\CM(B)$ be two absorbing $*$-homomorphisms that form a Cuntz pair.
	\begin{enumerate}[label=\textup{(\roman*)},leftmargin=*]
	\item \label{thm:original-KK-uniqueness:1}
	If $[\phi,\psi]=0$ in $KK(A,B)$, then $\phi$ and $\psi$ are strongly asymptotically unitarily equivalent. 
	\item \label{thm:original-KK-uniqueness:2}
	If $[\phi,\psi]=0$ in $KL(A,B)$, then there exists a sequence of unitaries $v_n\in\CU(\eins+B)$ such that
	\[
	\phi(a)=\lim_{n\to\infty} v_n\psi(a)v_n^*,\quad a\in A. 
	\]
	\end{enumerate}
\end{theorem}

%%%%%%%%%%%%%%%%%%%%%%%%%%%%%%%%%%%%%%%%%%%%%%%%%%%%%%%%%%%%%%%%%%

\section{On a generalized equivariant $KK$-group construction}

Throughout this section we continue following \autoref{standing-assumptions}.

\begin{defi}
Let $\FC$ be a set of cocycle representations from $(A,\alpha)$ to $(\CM(B),\beta)$.
We say that $\FC$ is \emph{downward closed}, if for every cocycle representation $(\phi,\Iu)$, we have that $(\phi,\Iu)\wc\FC$ implies $(\phi,\Iu)\in\FC$.
We say that $\FC$ is an \emph{eligible set} of cocycle representations for $(\alpha,\beta)$, if it is downward closed and contains every cocycle representation of the form $(0,\Iu)$, where $\Iu$ is an arbitrary $\beta$-cocycle in $\CM(B)$.
\end{defi}

\begin{rem}
Note that in the non-equivariant situation, i.e.\ $G=\{1\}$, being eligible is the same as being downward closed.
In general, if $\FC$ is a downward closed set as above, then it follows from \cite[Proposition 4.7]{GabeSzabo25} that $\FC$ is $\sigma$-additive in the sense of `Definition 4.13' therein, i.e., it is closed under countable sums as considered in \autoref{def:countable-sums}.
If $B$ is separable, then it follows by \cite[Theorem 4.16]{GabeSzabo25} that $\FC$ always has an \emph{absorbing} element, i.e., an element $(\phi,\Iu)\in\FC$ that absorbs every element of $\FC$ in the sense of \autoref{def:absorption}.
\end{rem}

\begin{defi}
Let $\FC$ be a set of cocycle representations from $(A,\alpha)$ to $(\CM(B),\beta)$.
We define $\FC[0,1]$ as the set of cocycle representations $(\Phi,\IU): (A,\alpha)\to (\CM(B[0,1]),\beta[0,1])$ such that $\ev_t\circ(\Phi,\IU)=:(\Phi_t,\IU_t)\in\FC$ for all $t\in [0,1]$, where $\ev_t: B[0,1]\to B$ is the evaluation map at $t$.
\end{defi}

\begin{defi} \label{def:EG-C}
Suppose that $\FC$ is an eligible set of cocycle representations from $(A,\alpha)$ to $(\CM(B),\beta)$.
An \emph{(equivariant) $(\FC;\alpha,\beta)$-Cuntz pair} is an equivariant $(\alpha,\beta)$-Cuntz pair consisting of cocycle representations $(\phi,\Iu)$ and $(\psi,\Iv)$ belonging to $\FC$.
We define $\IE^G(\FC;\alpha,\beta)$ to be the set of all such pairs and $\IE_0^G(\FC;\alpha,\beta)$ to be the set of all such pairs that are also anchored in the sense of \autoref{def:KKG-Thomsen}.

We say that two elements $x_0, x_1\in\IE^G(\FC;\alpha,\beta)$ are \emph{$\FC$-homotopic}, written $x_0\sim_{h,\FC} x_1$, if there exists $X\in\IE^G(\FC[0,1];\alpha,\beta[0,1])$ with $\ev_j\circ X=x_j$ for $j=0,1$.\footnote{For the sake of notational clarity we slightly abuse notation here. The Cuntz pair $\ev_j\circ X$ is meant to be the Cuntz pair coming from composing both components of $X$ from the left with the evaluation map in $j$.}
Given a pair $\big( (\phi,\Iu), (\psi,\Iv) \big)\in\IE^G(\FC;\alpha,\beta)$, we denote by $\big[ (\phi,\Iu), (\psi,\Iv) \big]$ its $\FC$-homotopy class.
As was the case before, the subset $\IE^G_0(\FC;\alpha,\beta)\subseteq \IE^G(\FC;\alpha,\beta)$ is closed under $\FC$-homotopy.
\end{defi}

\begin{rem} \label{rem:EG-C}
Exactly as one does in the Cuntz--Thomsen picture of equivariant $KK$-theory, one observes in the context of \autoref{def:EG-C} that Cuntz addition of $(\FC;\alpha,\beta)$-Cuntz pairs descends to a well-defined abelian group structure on $\IE^G(\FC;\alpha,\beta)/_{\sim_{h,\FC}}$.
The proof of this fact (see for instance \cite[Lemma 2.10]{GabeSzabo25}, though the argument goes back much further than that) carries over word for word from the (known) case of $\FC$ being the set of all cocycle representations.
Here the fact that $\FC$ is downward closed ensures that $\FC$ is closed under Cuntz addition and unitary equivalence, which enables the known proof to work here.
The fact that $\FC$ contains all cocycle representations of the form $(0,\Iu)$ for arbitrary $\beta$-cocycles $\Iu$ implies that there is a canonical homomorphism from the group $H_\beta$ (see \autoref{def:KKG}) to $\IE^G(\FC;\alpha,\beta)/_{\sim_{h,\FC}}$.
On the other hand, the assignment $[(\phi,\Iu),(\psi,\Iv)]\mapsto [(0,\Iu),(0,\Iv)]$ yields a well-defined homomorphism $\Ff: \IE^G(\FC;\alpha,\beta)/_{\sim_{h,\FC}}\to H_\beta$ that is a left inverse of the inclusion map $H_\beta\to\IE^G(\FC;\alpha,\beta)/_{\sim_{h,\FC}}$.
We call $\Ff$ the \emph{forgetful map} and observe that $\ker(\Ff)=\IE^G_0(\FC;\alpha,\beta)/_{\sim_{h,\FC}}$.
This is analogous to the observations made in \cite[Proposition 2.12]{GabeSzabo25}.
\end{rem}

\begin{defi} \label{def:generalized-KK-groups}
Suppose that $\FC$ is an eligible set of cocycle representations from $(A,\alpha)$ to $(\CM(B),\beta)$.
We define $KK^G(\FC;\alpha,\beta)$ to be the quotient of abelian groups $\big(\IE^G(\FC;\alpha,\beta)/_{\sim_{h,\FC}}\big)/H_\beta$.
By \autoref{rem:EG-C}, the forgetful map induces a direct sum decomposition $\IE^G(\FC;\alpha,\beta)/_{\sim_{h,\FC}}\cong H_\beta\oplus \IE_0^G(\FC;\alpha,\beta)/_{\sim_{h,\FC}}$ and therefore one has a natural isomorphism $KK^G(\FC;\alpha,\beta)\cong\IE_0^G(\FC;\alpha,\beta)/_{\sim_{h,\FC}}$.
\end{defi}

\begin{example} \label{ex:generalized-KK-examples}
Let us have a brief discussion in what way \autoref{def:generalized-KK-groups} extends various known constructions.
\begin{enumerate}[leftmargin=*,label=\textup{(\arabic*)}]
	\item If $\FC$ denotes the set of all cocycle representations, then $KK^G(\FC;\alpha,\beta)$ agrees with the equivariant $KK$-group $KK^G(\alpha,\beta)$ on the nose in Thomsen's description \cite{Thomsen98}.
	\item For $G=\{1\}$, we may choose as $\FC$ the set of all weakly nuclear representations $A\to\CM(B)$.
Then $KK(\FC;A,B)$ is naturally isomorphic to Skandalis' nuclear $KK$-group $KK_\nuc(A,B)$ \cite{Skandalis88}; see \cite[Section 12]{Gabe24} for a rigorous treatment of the Cuntz picture in that setting.
	\item Still assuming $G=\{1\}$, Kirchberg's ideal-related $KK$-groups fall under the above framework.
	We refer to \cite[Sections 10--12]{Gabe24} for a rigorous treatment of the concepts mentioned here.
	For a topological space $X$, recall that an \emph{action} of $X$ on a \cstar-algebra $A$ consists of an order preserving map $\Phi_A: \CO(X)\to\CI(A)$, where the domain is the lattice of open subsets of $X$ and the codomain is the ideal lattice of $A$.
	For an open set $U\subseteq X$, one writes $A(U)=\Phi_A(U)$ as short-hand.
	A \cstar-algebra with a given action of $X$ is called an \emph{$X$-\cstar-algebra}.
	If we are given $X$-\cstar-algebras $(A,\Phi_A)$ and $(B,\Phi_B)$, then a $*$-homomorphism $\phi: A\to\CM(B)$ is called \emph{weakly $X$-equivariant} if for all $b\in B$, $a\in A$ and every open set $U\subseteq X$ one has $b^*\phi(A(U))b\subseteq B(U)$.
	It is immediate that if $A$ is separable and $B$ is stable and $\sigma$-unital, then the set $\FC_X$ of all weakly $X$-equivariant representations $A\to\CM(B)$ is an eligible set.
	Upon comparing with the literature, we see that the group $KK(\FC_X;A,B)$ in the sense of \autoref{def:generalized-KK-groups} agrees with Kirchberg's group $KK(X;A,B)$; see \cite[Proposition 2.28]{Gabe24}.
	The same is true if we consider the subset of representations $\FC_X^{\nuc}\subseteq\FC_X$ that are additionally weakly nuclear, in which case we obtain the nuclear ideal-related group $KK_\nuc(X;A,B)$.
	\item Gabe's article \cite{Gabe24} includes a more detailed discussion about the $KK$-theory of $\CC_0(X)$-algebras when $X$ is locally compact Hausdorff, which can be interpreted as the $KK$-theory of $X$-\cstar-algebras.
	From this point of view, if now $G\neq\{1\}$ and $X$ is compact Hausdorff, then Kasparov's notion \cite{Kasparov88} of $KK$-theory for $G$-$\CC(X)$-algebras can also be recast and generalized in the language of \autoref{def:generalized-KK-groups}.
	Since this is as straightforward to write out as it is lengthy and technical to justify in all its rigour, we omit further details here. 
\end{enumerate}
\end{example}

\begin{defi}[{\cite[Definition 3.5]{GabeSzabo25}}]
Two given cocycle representations $(\phi,\Iu), (\psi,\Iv): (A,\alpha)\to (\CM(B),\beta)$ are called \emph{operator homotopic}, if there exists a unitary $W\in\CU_0(\SD_{(\phi,\Iu)})$ such that $(\phi,\Iu)=\ad(W)\circ(\psi,\Iu)$.
\end{defi}

\begin{rem} \label{rem:op-hom-implies-hom}
Let $\FC$ be an eligible set.
Suppose that two given elements $(\phi,\Iu), (\psi,\Iv)\in\FC$ are operator homotopic.
Just as it was the case for ordinary equivariant $KK$-theory, it follows that the pair $\big( (\phi,\Iu), (\psi,\Iv) \big)$ is an $(\FC;\alpha,\beta)$-Cuntz pair that is $\FC$-homotopic to the degenerate pair $\big( (\phi,\Iu), (\phi,\Iu) \big)$ by connecting the chosen unitary $W$ above to the unit.
Hence $[ (\phi,\Iu), (\psi,\Iv) ]=0$ in $KK^G(\FC;\alpha,\beta)$.
This merely uses the fact that $\FC$ is closed under unitary equivalence.
\end{rem}

In complete analogy to what was observed in \cite[Section 3]{GabeSzabo25} as a Cuntz--Thomsen picture analog of Kasparov's well-known observations about operator homotopy, we shall show in the rest of this section that the implication in \autoref{rem:op-hom-implies-hom} has a stable converse.
That is, if some equivariant $(\FC;\alpha,\beta)$-Cuntz pair is $\FC$-homotopic to a degenerate one, then the two underlying cocycle representations are operator homotopic after adding a suitable representation in $\FC$.

\begin{nota}
Assume that $\beta$ is a strongly stable action, witnessed by a sequence of isometries $r_n\in\CM(B)^\beta$ with $\sum_{n=1}^\infty r_n r_n^*=\eins$.
Let us choose a family of matrix units $\{ e_{k,\ell} \mid k,\ell\geq 1\}$ generating the \cstar-algebra $\CK$ of compact operators on $\ell^2(\IN)$.
Below, we denote by $\Lambda$ the equivariant isomorphism $\Lambda: (\CK\otimes B,\id_\CK\otimes\beta)\to (B,\beta)$ given by $\Lambda(e_{k,\ell}\otimes b)=r_kbr_\ell^*$ for every $b\in B$ and $k,\ell\geq 1$.
\end{nota}

\begin{lemma} \label{lem:coc-rep-compositions}
Assume that $\FC$ is a downward closed set of cocycle representations from $(A,\alpha)$ to $(\CM(B),\beta)$.
Let $(\Omega,\IX)\in\FC[0,1]$.
If $\theta: \CC[0,1]\to\CM(\CK)\cong\IB(\ell^2(\IN))$ is any faithful unital representation with $\theta(\CC[0,1])\cap\CK=\{0\}$, then the composition of
\[
(A,\alpha) \stackrel{(\Omega,\IX)}{\longrightarrow} (\CM(B[0,1]),\beta[0,1]) \stackrel{\theta\otimes\id_B}{\longrightarrow} (\CM(\CK\otimes B),\id_\CK\otimes\beta) \stackrel{\Lambda}{\longrightarrow} (\CM(B),\beta)
\]
is a cocycle representation in $\FC$.
\end{lemma}
\begin{proof}
Let $(\kappa,\Ix)$ be the cocycle representation that arises from the composition above.
Since $\FC$ is downward closed, it suffices to verify $(\kappa,\Ix)\wc\FC$.

We first claim that it suffices to prove this for a particular choice of $\theta$.
Indeed, if $\theta_1$ and $\theta_2$ are two representations as in the statement, then denote by $(\kappa^{(j)},\Ix^{(j)})$ their associated compositions for $j=1,2$.
It follows by Voiculescu's theorem \cite{Voiculescu76} that there exists a sequence of unitaries $U_n\in\CU(\CM(\CK))$ such that $\ad(U_n)\circ\theta_1\to \theta_2$ in the point-norm topology.
The sequence of unitaries $W_n=\Lambda(U_n\otimes\eins_{\CM(B)})\in\CM(B)^\beta$ is then seen to satisfy the limit behavior
\[
\begin{array}{ccl}
\ad(W_n)\circ\kappa^{(1)} &=& \Lambda\circ \ad(U_n\otimes\eins_{\CM(B)})\circ (\theta_1\otimes\id_B)\circ\Omega \\
&\to& \Lambda\circ (\theta_2\otimes\id_B)\circ\Omega \ = \ \kappa^{(2)}
\end{array} 
\]
in the point-strict topology, as well as 
\[
\begin{array}{ccl}
W_n\Ix^{(1)}_g\beta_g(W_n)^* &=& W_n\Ix_gW_n^* \\
&=& \Lambda(U_n\otimes\eins_{\CM(B)}) \Lambda((\theta_1\otimes\id_B)(\IX_g)) \Lambda(U_n^*\otimes\eins_{\CM(B)}) \\
&\to& \Lambda((\theta_2\otimes\id_B)(\IX_g)) \ = \ \Ix^{(2)}_g
\end{array} 
\]
strictly and uniformly over compact subsets of $G$.
By \cite[Proposition 4.7]{GabeSzabo25}, it follows that $(\kappa^{(1)},\Ix^{(1)})$ is weakly contained in $(\kappa^{(2)},\Ix^{(2)})$.
We conclude that if the claim holds for $\theta_2$ in place of $\theta$, then it also holds for $\theta_1$ as $\FC$ was downward closed.
Since $\theta_1$ and $\theta_2$ were both arbitrary, we have verified that it suffices to prove the statement for a specific choice of $\theta$.

Let us choose a surjective map $\IN\to\IQ\cap[0,1]$, $n\mapsto t_n$.
Let $(\delta_n)_{n\geq 1}$ be the standard orthonormal basis for $\CH=\ell^2(\IN)$ and define $\theta: \CC[0,1]\to\CM(\CK)=\CB(\CH)$ via $\theta(f)(\xi)=\sum_{n=1}^\infty f(t_n)\langle \delta_n\mid\xi\rangle\delta_n$ for all $f\in\CC[0,1]$ and $\xi\in\ell^2(\IN)$.\footnote{For this formula we adopt the convention that the inner product is linear in the second component.}
We proceed to prove the statement for this choice of $\theta$.

Let $\eps>0$ be a constant, $b\in B$ a contraction, $\CF\subset A$ a finite set and $K\subseteq G$ a compact subset.
We may assume that $\CF$ consists of contractions.
We may choose a number $N$ and a positive contraction $e\in B$ such that $\|b(\eins-b_0)\|\leq\eps/2$, where $b_0=\sum_{n=1}^N r_n e r_n^*$.
Let $P_N\in\CK\subseteq\CM(\CK\otimes B)$ be given as $P_N=\sum_{n=1}^N e_{n,n}$.
Then we have for all $f\in\CC[0,1]$ that $P_N\theta(f)P_N=\sum_{n=1}^N f(t_n)e_{n,n}$.
Hence we have for all $a\in A$ that
\[
\begin{array}{ccl}
b_0^*\kappa(a)b_0 &=& \Lambda(P_N\otimes e) \kappa(a) \Lambda(P_N\otimes e) \\
&=& \Lambda\Big[ (P_N\otimes e) (\theta\otimes\id_B)(\Omega(a)) (P_N\otimes e)\Big] \\
&=& \Lambda\Big[ \sum_{n=1}^N e_{n,n}\otimes (e\Omega_{t_n}(a)e) \Big] \\
&=& \dst\sum_{n=1}^N r_n e \Omega_{t_n}(a) e r_n^*.
\end{array}
\]
Furthermore we have for all $g\in G$ that
\[
\begin{array}{ccl}
b_0^*\Ix_g\beta_g(b_0) &=& b_0^*\cdot\Lambda((\theta\otimes\id_B)(\IX_g)) \cdot\beta_g(b_0) \\
&=& \Lambda\Big[ (P_N\otimes e) (\theta\otimes\id_B)(\IX_g) (P_N\otimes \beta_g(e))\Big] \\
&=& \Lambda\Big[ \sum_{n=1}^N e_{n,n}\otimes (e\IX_{g,t_n}\beta_g(e)) \Big] \\
&=& \dst\sum_{n=1}^N r_ne \IX_{g,t_n}\beta_g(e)r_n^*.
\end{array}
\]
If we consider $c_j=er_n^*b$ for all $j=1,\dots,N$, then we observe for all $a\in\CF$ that
\[
b^*\kappa(a)b =_{\eps} b^*b_0^*\kappa(a)b_0b =  \sum_{n=1}^N c_n^*\Omega_{t_n}(a)c_n
\]
and for all $g\in K$ that
\[
b^*\Ix_g\beta_g(b)^* =_\eps b^*b_0^*\Ix_g\beta_g(b_0b)=\sum_{n=1}^N c_n^*\IX_{g,t_n}\beta_g(c_n).
\]
Since $(\Omega_t,\IU_t)\in\FC$ for all $t\in [0,1]$ and the quadruple $(\eps,b,\CF,K)$ was arbitrary, we get that indeed $(\kappa,\Ix)\wc\FC$.
\end{proof}

\begin{lemma} \label{lem:KKarantine}
Let $\FC$ be a downward closed set of cocycle representations from $(A,\alpha)$ to $(\CM(B),\beta)$.
Suppose that $\big( (\Phi,\IU), (\Psi,\IV) \big)$ is a $(\FC[0,1]; \alpha, \beta[0,1])$-Cuntz pair such that 
\[
(\phi,\Iu) :=\ev_0\circ(\Phi,\IU)=\ev_1\circ(\Phi,\IU)=\ev_0\circ(\Psi,\IV).
\]
Define $(\psi,\Iv):=\ev_1\circ(\Psi,\IV)$.
Then there exists a cocycle representation $(\kappa,\Ix)\in\FC$ such that $(\phi, \Iu) \oplus (\kappa, \Ix)$ and $(\psi,\Iv) \oplus (\kappa, \Ix)$ are operator homotopic.
\end{lemma}
\begin{proof}
We apply the construction and argument as described in the statements of \cite[Notation 3.7, Lemma 3.8]{GabeSzabo25}.
This gives a cocycle representation $(\kappa,\Ix)$ such that $(\phi, \Iu) \oplus (\kappa, \Ix)$ and $(\psi,\Iv) \oplus (\kappa, \Ix)$ are operator homotopic.
Upon a close inspection of how $(\kappa,\Ix)$ is constructed there\footnote{This is admittedly quite lengthy to check when reading, but we appeal to the construction described in \cite[Notation 3.7]{GabeSzabo25} as-is. We omit the details here, as we would otherwise just copy content from this reference. Notably, the proof of ``Lemma 3.8'' therein is irrelevant for this discussion.}, this cocycle representation arises as a Cuntz sum of $(\phi,\Iu)$ and a composition of cocycle representations that is of the same form as in \autoref{lem:coc-rep-compositions}. 
To be more precise, if one starts with a representation $\theta_0: \CC[0,1]\to\CM(\CK(\CH))$ satisfying the conclusion of \cite[Theorem 2.15]{GabeSzabo25}, then one sets
\[
\theta=\theta_0\oplus\theta_0: \CC[0,1]\to\CM(\CK(\CH\oplus\CH))
\] 
(as an ordinary direct sum) and 
\[
(\Omega,\IX)=(\Psi,\IV)\oplus (\Phi,\IU): (A,\alpha)\to (\CM(B[0,1]),\beta[0,1])
\]
(as a Cuntz sum).
Let $(\chi,\Iy)$ be the cocycle representation that arises as a composition as in \autoref{lem:coc-rep-compositions} with these choices made.
Evidently $(\Omega,\IX)\in\FC[0,1]$ and thus also $(\chi,\Iy)\in\FC$.
The construction and conclusion of \cite[Notation 3.7, Lemma 3.8]{GabeSzabo25} then tells us that $(\kappa,\Ix)=(\phi,\Iu)\oplus(\chi,\Iy)\in\FC$ satisfies the desired property. 
\end{proof}

We close this section by including a straightforward generalization of a known observation for  $KK$-theory:

\begin{prop} \label{prop:standard-form-KK-classes}
Let $\FC$ be an eligible set.
Let $(\theta,\Iy)\in\FC$ be an absorbing element.\footnote{We remind that reader that if $B$ is separable, such an element always exists by \cite[Theorem 4.16]{GabeSzabo25}.}
Then:
\begin{enumerate}[label=\textup{(\roman*)},leftmargin=*] 
\item Every element $x\in \IE^G(\FC;\alpha,\beta)/{\sim_{h,\FC}}$ can be expressed as the equivalence class of a $(\FC;\alpha,\beta)$-Cuntz pair $\big( (\phi,\Iu), (\theta,\Iy) \big)$ for some absorbing element $(\phi,\Iu)\in\FC$. \label{prop:standard-form-KK-classes:1}
\item Every element $z\in KK^G(\FC;\alpha,\beta)$ can be expressed as the equivalence class of an anchored $(\FC;\alpha,\beta)$-Cuntz pair $\big( (\phi,\Iu), (\theta,\Iy) \big)$ for some absorbing element $(\phi,\Iu)\in\FC$. \label{prop:standard-form-KK-classes:2}
\end{enumerate}
\end{prop}
\begin{proof}
This is the generalization of \cite[Corollary 4.17]{GabeSzabo25}.
We can copy the proof therein word for word to prove the statement here.
The argument carries over because elements of $\FC$ are closed under unitary equivalence and Cuntz sums.
\end{proof}

%%%

\section{$KK$-uniqueness}

\begin{nota} \label{nota:C-phi}
Let $\alpha: G\curvearrowright A$ and $\beta: G\curvearrowright B$ be actions on \cstar-algebras.
Let $(\phi,\Iu): A\to \CM(B)$ be a cocycle representation.
We define 
\[
\SC_{(\phi,\Iu)}=\set{ x\in\CM(B) \mid x\phi(A)\cup\phi(A)x\subseteq B, \set{x-\beta^\Iu_g(x)}_{g\in G}\subseteq B}.
\]
When we compare with \autoref{nota:D-phi}, we see that $\SC_{(\phi,\Iu)}\triangleleft\SD_{(\phi,\Iu)}$ is an ideal.
We may thus consider the quotient $\SB_{(\phi,\Iu)}=\SD_{(\phi,\Iu)}/\SC_{(\phi,\Iu)}$.

Given an arbitrary $*$-homomorphism $\psi: A\to\CM(B)$, we may view it as a cocycle representation with respect to $G=\{1\}$ and hence write $\SC_\psi$ as short-hand for $\SC_{(\psi,\eins)}$.
Using this notation, we can also observe that when $(\phi,\Iu)$ is given as above, then $\SC_\phi$ is genuinely $\beta^\Iu$-invariant and $\SC_\phi\triangleleft\SD^w_{(\phi,\Iu)}$ is an ideal.
\end{nota}

The following lemma and many variations of it are well-known, but we include a proof for the reader's convenience.

\begin{lemma} \label{lem:isometry-map}
Let $E$ be any \cstar-algebra and $\theta: E\to E$ an endomorphism.
Suppose that there exists an isometry $S\in\CM(E)$ with $\theta(x)=SxS^*$ for all $x\in E$.
Then one has for every ideal $I\triangleleft E$ that $\theta(I)\subseteq I$.
For any inclusion of ideals $J\triangleleft I\triangleleft E$, denote by $\theta_J^I$ the induced endomorphism on $I/J$ via $\theta_J^I(x+J)=\theta(x)+J$ for all $x\in I$.
It follows that $K_*(\theta^I_J)=\id$ on $K_*(I/J)$.
\end{lemma}
\begin{proof}
Given ideals $J\triangleleft I\triangleleft E$, it follows from the universal property of the multiplier algebra that there is a canonical $*$-homomorphism $\pi: \CM(E)\to\CM(I/J)$ given by $\pi(y)(b+J)=yb+J$ for all $b\in I$ and $y\in\CM(E)$.
Hence $\pi(S)\in\CM(I)$ is an isometry implementing $\theta^I_J$.
Therefore, it suffices to show the last part of the claim for $I=E$ and $J=0$.

Let $\{e_{k\ell}\}_{k,\ell=1,2}$ be the generating matrix units of the $2\times 2$ matrix algebra $M_2$.
In order to prove the claim it suffices (by stability of $K$-theory) to show that the composition with the corner embedding $e_{11}\otimes\theta: E\to M_2(E)$ is homotopic to $e_{11}\otimes\id_E$.
Choose some norm-continuous path of unitaries $W: [0,1]\to M_2\subseteq\CM(M_2(E))$ such that $W_0=\eins$ and $W_1=e_{21}+e_{12}$ (e.g., $W_t=\exp(\frac{\pi}{2} i t (\eins-W_1))$).
The resulting continuous path of isometries $R_t=W_t(e_{11}\otimes S+ e_{22}\otimes\eins)W_t^*$ satisfies $R_0=(e_{11}\otimes S+ e_{22}\otimes\eins)$ and $R_1=(e_{11}\otimes\eins+ e_{22}\otimes S)$.
The point-norm continuous path of embeddings $\theta_t: E\to M_2(E)$ for $t\in [0,1]$ given by
$\theta_t(x)=R_t(e_{11}\otimes x)R_t^*$ then satisfies $\theta_0=e_{11}\otimes\theta$ and $\theta_1=e_{11}\otimes\id_E$.
\end{proof}

We remind the reader of the standing assumptions from \autoref{standing-assumptions}, which still applies to this section and the subsequent section.

\begin{lemma} \label{lem:weak-absorption-K-theory}
Suppose that $(\phi,\Iu), (\psi,\Iv): (A,\alpha)\to(\CM(B),\beta)$ are two cocycle representations and that $(\phi,\Iu)$ weakly absorbs $(\psi,\Iv)$.
Consider the embedding 
\[
\id\oplus 0: \SD^w_{(\phi,\Iu)}\to \SD^w_{(\phi,\Iu)\oplus(\psi,\Iv)} \text{ given by } x\mapsto x\oplus 0 \text{ (Cuntz addition).}
\]
Then $K_*(\id\oplus 0)$ is an isomorphism of abelian groups.
Furthermore $\id\oplus 0$ descends to embeddings $\SD_{(\phi,\Iu)}\to \SD_{(\phi,\Iu)\oplus(\psi,\Iv)}$ and $\SC_{(\phi,\Iu)}\to \SC_{(\phi,\Iu)\oplus(\psi,\Iv)}$ that also induce isomorphisms in $K$-theory.
\end{lemma}
\begin{proof}
Let $r_1,r_2\in\CM(B)^\beta$ be two isometries with $r_1r_1^*+r_2r_2^*=\eins$.
It shall be understood that Cuntz sums in this proof are formed with this pair.
By assumption, there exists a unitary $U\in\CM(B)$ such that $\ad(U)\circ(\phi\oplus\psi,\Iu\oplus\Iv)$ agrees with $(\phi,\Iu)$ modulo $B$.
The conjugation map $\kappa=\ad(Ur_1)$ then induces an embedding that is a composition
\[
\kappa: \SD_{(\phi,\Iu)} \stackrel{\ad(r_1)}{\longrightarrow} \SD_{(\phi\oplus\psi,\Iu\oplus\Iv)} \stackrel{\ad(U)}{\longrightarrow} \SD_{(\phi,\Iu)}.
\]
Since we have for all $a\in A$ that
\[
\phi(a)Ur_1 \equiv_B U(\phi\oplus\psi)(a)r_1 = Ur_1\phi(a) 
\]
and for all $g\in G$ that
\[
\beta_g^\Iu(Ur_1)= \Iu_g\beta_g(U) r_1 \Iu_g^* \equiv_B U (\Iu_g\oplus\Iv_g) r_1 \Iu_g^* = Ur_1,
\]
it follows that $Ur_1\in\SD_{(\phi,\Iu)}\subseteq\SD_{(\phi,\Iu)}^w$.
This element is obviously an isometry.
It follows from \autoref{lem:isometry-map} that $\kappa$ induces isomorphisms on $K$-theory.
Since the map $\ad(U)$ above is a genuine isomorphism (\autoref{rem:D-infinite-repeat}), it follows that $\ad(r_1)=\id\oplus 0: \SD_{(\phi,\Iu)}\to \SD_{(\phi,\Iu)\oplus(\psi,\Iv)}$ induces an isomorphism on $K$-theory as well.
It persists at the level of ideals such as the inclusion $\SC_{(\phi,\Iu)}\to \SC_{(\phi,\Iu)\oplus(\psi,\Iv)}$.
The same follows for $\SD_{(\phi,\Iu)}^w$.
\end{proof}

The next result is the first of two Paschke-type duality results, as discussed in the introduction.
A more refined version of a Paschke duality result, which is of interest mainly for \cstar-dynamics, will be explored in the next section.

\begin{theorem} \label{thm:Paschke-1}
Let $\alpha: G\curvearrowright A$ be an action on a separable \cstar-algebra and $\beta: G\curvearrowright B$ a strongly stable action on a $\sigma$-unital \cstar-algebra.
Let $\FC$ be an eligible set of cocycle representations $(A,\alpha)\to (\CM(B),\beta)$, and let $(\Phi,\IU)\in\FC$ be an absorbing element.
Given any $U\in\CU(\SD_{(\Phi,\IU)})$, we consider the Cuntz pair $\xi(U)\in \IE^G(\FC;\alpha,\beta)$ given by $\ad(U)\circ(\Phi,\IU)$ and $(\Phi,\IU)$.
Then the assignment $\Xi: K_1(\SD_{(\Phi,\IU)})\to \IE^G(\FC;\alpha,\beta)/_{\sim_{\FC,h}}$ given by $\Xi([U])=[\xi(U)]$ defines an isomorphism of abelian groups.
\end{theorem}
\begin{proof}
Let us first consider the map $\Omega: \CU(\SD_{(\Phi,\IU)})\to \IE^G(\FC;\alpha,\beta)/_{\sim_{\FC,h}}$ given by $\Omega(U)=[\xi(U)]$.
By \autoref{thm:K1-injectivity}, $\SD_{(\Phi,\IU)}$ is a properly infinite $K_1$-injective \cstar-algebra, hence the canonical map $\CU(\SD_{(\Phi,\IU)})/\CU_0(\SD_{(\Phi,\IU)})\to K_1(\SD_{(\Phi,\IU)})$ is an isomorphism.
Therefore it suffices to show that $\Omega$ is a well-defined surjective group homomorphism with kernel $\CU_0(\SD_{(\Phi,\IU)})$.

The fact that $\xi$ and therefore also $\Omega$ is well-defined is due to the definition of $\SD_{(\Phi,\IU)}$.
For any $U\in\CU(\SD_{(\Phi,\IU)})$, the cocycle representation $\ad(U)\circ(\Phi,\IU)$ belongs to $\FC$ and is equal to $(\Phi,\IU)$ modulo $B$, thus indeed $\xi(U)\in\IE^G(\FC;\alpha,\beta)$.
The fact that $\Omega$ is a group homomorphism follows from the following calculation for $U,V\in\CU(\SD_{(\Phi,\IU)})$:
\[
\begin{array}{ccl}
\Omega(UV) &=& \big[ \ad(UV)\circ(\Phi,\IU), (\Phi,\IU) \big] \\
&=& \big[ \ad(UV)\circ(\Phi,\IU), \ad(U)\circ(\Phi,\IU) \big]+\big[ \ad(U)\circ(\Phi,\IU), (\Phi,\IU) \big] \\
&=& \big[ \ad(V)\circ(\Phi,\IU), (\Phi,\IU) \big]+\big[ \ad(U)\circ(\Phi,\IU), (\Phi,\IU) \big]\\
&=& \Omega(U)+\Omega(V).
\end{array}
\]
For the surjectivity, let $\big( (\psi,\Iv), (\theta,\Iy) \big)\in \IE^G(\FC;\alpha,\beta)$ be a Cuntz pair.
We shall employ an argument analogous to the proof of \cite[Corollary 4.17]{GabeSzabo25}.
Since $(\Phi,\IU)$ is an absorbing element of $\FC$, we have equivalences
\[
(\Phi,\IU)\oplus(\theta,\Iy) \asymp (\Phi,\IU)\asymp (\Phi,\IU)\oplus(\psi,\Iv).
\]
Let $V: [0,\infty)\to\CU(\CM(B))$ be a norm-continuous path satisfying
\begin{itemize}
  	\item $\Phi(a) = \dst\lim_{t\to\infty} V_t (\Phi(a)\oplus\theta(a)) V_t^*$ for all $a\in A$;
  	\item $\Phi(a) - V_t (\Phi(a)\oplus\theta(a)) V_t^* \in B$ for all $a\in A$ and $t\geq 0$;
  	\item $\dst\lim_{t\to\infty} \max_{g\in K}\ \| \IU_g - U_t (\IU_g\oplus\Iy_g) \beta_g(U_t)^* \| =0$ for all compact sets $K\subseteq G$;
  	\item $\IU_g - U_t (\IU_g\oplus\Ix_g) \beta_g(U_t)^* \in B$ for all $t\geq 0$ and $g\in G$.
  \end{itemize} 
From these properties we can read off that the degenerate $(\FC;\alpha,\beta)$-Cuntz pair $\big( (\Phi,\IU), (\Phi,\IU) \big)$ is $\FC$-homotopic to $\big( \ad(V_0)\circ(\Phi\oplus\theta,\IU\oplus\Iy), (\Phi,\IU) \big)$.
Analogously, we find a norm-continuous path $W: [0,\infty)\to\CU(\CM(B))$ such that $\ad(W_t)\circ(\Phi\oplus\psi,\IU\oplus\Iv)$ converges to $(\Phi,\IU)$ in the same manner.
Note that since $(\theta,\Iy)$ and $(\psi,\Iv)$ agree modulo $B$, the same is true for
\[
\begin{array}{cll}
(\Phi,\IU) &\equiv_B & \ad(V_0)\circ(\Phi\oplus\theta,\IU\oplus\Iy) \\
&\equiv_B& \ad(V_0)\circ(\Phi\oplus\psi,\IU\oplus\Iv) \\
&\equiv_B& \ad(V_0W_0^*)\circ(\Phi,\IU).
\end{array}
\]
This means that the unitary $V_0W_0^*$ belongs to $\SD_{(\Phi,\IU)}$.
We compute in $\IE^G(\FC;\alpha,\beta)/_{\sim_{\FC,h}}$ that
\[
\begin{array}{ccl}
\big[ (\psi,\Iv), (\theta,\Iy) \big] &=& \big[ (\Phi\oplus\psi,\IU\oplus\Iv), (\Phi\oplus\theta,\Phi\oplus\Iy) \big] \\
&=& \big[ \ad(W_0^*)\circ(\Phi,\IU), (\Phi\oplus\psi,\IU\oplus\Iv) \big]  \\
&& +\big[ (\Phi\oplus\psi,\IU\oplus\Iv), (\Phi\oplus\theta,\Phi\oplus\Iy) \big] \\
&& + \big[ (\Phi\oplus\theta,\Phi\oplus\Iy), \ad(V_0^*)\circ(\Phi,\IU) \big] \\
&=& \big[ \ad(W_0^*)\circ(\Phi,\IU), \ad(V_0^*)\circ(\Phi,\IU) \big] \\
&=& \big[ \ad(V_0W_0^*)\circ(\Phi,\IU), (\Phi,\IU) \big] \ = \ \Omega(V_0W_0^*).
\end{array}
\]
Since the pair $\big( (\psi,\Iv), (\theta,\Iy) \big)\in \IE^G(\FC;\alpha,\beta)$ was arbitrary, we see that $\Omega$ is surjective.

Lastly, let us determine the kernel of $\Omega$.
By \autoref{rem:op-hom-implies-hom} we have that $\CU_0(\SD_{(\Phi,\IU)})$ is contained in the kernel.
We claim equality.
Suppose that $U\in\CU(\SD_{(\Phi,\IU)})$ is given with $\Omega(U)=0$.
This means that the $(\FC;\alpha,\beta)$-Cuntz pair $\big( \ad(U)\circ(\Phi,\IU), (\Phi,\IU) \big)$ is $\FC$-homotopic to the degenerate pair $\big( (\Phi,\IU), (\Phi,\IU) \big)$.
By \autoref{lem:KKarantine}, it follows that there exists an element $(\kappa,\Ix)\in\FC$ such that $(\Phi\oplus\kappa,\IU\oplus\Ix)$ and $\ad(U\oplus\eins)\circ(\Phi\oplus\kappa,\IU\oplus\Ix)$ are operator homotopic.
Without loss of generality, by adding more to $(\kappa,\Ix)$ if necessary, we may assume that $(\Phi\oplus\kappa,\IU\oplus\Ix)$ is an infinite repeat.
There exists $W\in\CU_0(\SD_{(\Phi\oplus\kappa,\IU\oplus\Ix)})$ such that
\[
\ad(W(U\oplus \eins))\circ(\Phi\oplus\kappa,\IU\oplus\Ix) = (\Phi\oplus\kappa,\IU\oplus\Ix).
\]
In other words, we have that $W(U\oplus\eins)\in \big( \CM(B)\cap(\Phi\oplus\kappa)(A)' \big)^{\beta^{\IU\oplus\Ix}}$.
By \autoref{prop:infinite-repeat}, the $K_1$-class of $W(U\oplus\eins)$ in $\SD_{(\Phi\oplus\kappa,\IU\oplus\Ix)}$ is therefore trivial.
As $W$ is homotopic to the unit, we have $[U\oplus\eins]=0$ in $K_1(\SD_{(\Phi\oplus\kappa,\IU\oplus\Ix)})$.
By \autoref{lem:weak-absorption-K-theory}, we must have $[U]=0$ in $K_1(\SD_{(\Phi,\IU)})$ and hence $U\in\CU_0(\SD_{(\Phi,\IU)})$ with \autoref{thm:K1-injectivity}.
This finishes the proof.
\end{proof}

Before we state our main result, we state an important technical lemma that we will use in its proof.

\begin{lemma}[see {\cite[Corollary 5.4]{GabeSzabo25}}] \label{lem:saue}
Let $\alpha: G\curvearrowright A$ and $\beta: G\curvearrowright B$ be two actions on \cstar-algebras such that $A$ is separable and $B$ is $\sigma$-unital.
Let $(\phi,\Iu), (\psi,\Iv): (A,\alpha)\to (\CM(B),\beta)$ be two cocycle representations.
Suppose that there exists a norm-continuous path of unitaries $U: [0,\infty)\to\CU(\SD_{(\phi,\Iu)})$ with $U_0=\eins$ and
\[
\lim_{t\to\infty} \Big( \|\psi(a)-U_t\phi(a)U_t^*\| + \max_{g\in K} \|\Iv_g-U_t\Iu_g\beta_g(U_t)^*\| \Big) = 0
\]
for all $a\in A$ and every compact set $K\subseteq G$.
Then $(\phi,\Iu)$ and $(\psi,\Iv)$ are strongly asymptotically unitarily equivalent.
\end{lemma}

This brings us to the main result of the article, which is a generalized dynamical $KK$-uniqueness theorem that improves upon the dynamical stable uniqueness theorem \cite[Theorem 6.4]{GabeSzabo25}.
It also entails the uniqueness theorem for nuclear and/or ideal-related $KK$-theory as treated in \cite[Theorem 13.9]{Gabe24}.
The content of this section also represents a somewhat more self-contained approach to answer \cite[Question 5.17]{CGSTW23} in the positive, as well as an assortment of conceivable variants of that question for other versions of $KK$-theory.
Restricting its statement to the case of separable \cstar-algebras, one obtains \autoref{theorem:C}.

\begin{theorem} \label{thm:KK-uniqueness}
Let $G$ be a second-countable, locally compact group.
Let $\alpha: G\curvearrowright A$ be an action on a separable \cstar-algebra and $\beta: G\curvearrowright B$ a strongly stable action on a $\sigma$-unital \cstar-algebra.
Let $\FC$ be an eligible set of cocycle representations $(A,\alpha)\to (\CM(B),\beta)$.
Let $(\phi,\Iu), (\psi,\Iv)\in\FC$ be two absorbing elements that form an anchored $(\FC;\alpha,\beta)$-Cuntz pair.
Then $[(\phi,\Iu),(\psi,\Iv)]=0$ in $KK^G(\FC;\alpha,\beta)$ if and only if $(\phi,\Iu)$ and $(\psi,\Iv)$ are strongly asymptotically unitarily equivalent.
\end{theorem}
\begin{proof}
Since the ``if'' part is trivial, we only need to prove the ``only if'' part.
Assume $[(\phi,\Iu),(\psi,\Iv)]=0$ in $KK^G(\FC;\alpha,\beta)$.
Since the pair was anchored, this is equivalent to $[(\phi,\Iu),(\psi,\Iv)]=0$ in the group $\IE^G(\FC;\alpha,\beta)/_{\sim_{\FC,h}}$.
Since $(\phi,\Iu)$ and $(\psi,\Iv)$ were both absorbing elements of $\FC$, we have 
\[
(\phi,\Iu)\asymp(\phi,\Iu)\oplus(\psi,\Iv)\asymp(\psi,\Iv).
\]
Choose a norm-continuous path $U: [1,\infty)\to\CU(\CM(B))$ satisfying
\begin{enumerate}[leftmargin=5mm,label={$\bullet$}]
  	\item $\phi(a) = \dst\lim_{t\to\infty} U_t \psi(a) U_t^*$ for all $a\in A$;
  	\item $\phi(a) - U_t \psi(a) U_t^* \in B$ for all $a\in A$ and $t\geq 1$;
  	\item $\dst\lim_{t\to\infty} \max_{g\in K}\ \| \Iu_g - U_t \Iv_g \beta_g(U_t)^* \| =0$ for all compact sets $K\subseteq G$;
  	\item $\Iu_g - U_t \Iv_g \beta_g(U_t)^* \in B$ for all $t\geq 1$ and $g\in G$.
  \end{enumerate} 
Given that $(\phi,\Iu)$ and $(\psi,\Iv)$ agree modulo $B$, the second and the fourth items imply that $U_t\in\SD_{(\phi,\Iu)}$ for all $t\geq 1$.
This path induces the equality $0=[(\phi,\Iu),(\psi,\Iv)]=[\ad(U_1)\circ(\psi,\Iv),(\psi,\Iv)]$.
In light of \autoref{thm:Paschke-1}, this means $0=[U_1]$ in $K_1(\SD_{(\phi,\Iu)})$, which by \autoref{thm:K1-injectivity} means that $U_1\in\CU_0(\SD_{(\Phi,\IU)})$.
We may thus extend the given path $(U_t)_{t\geq 1}$ to a norm-continuous path $U: [0,\infty)\to\CU(\SD_{(\phi,\Iu)})$ with $U_0=\eins$.
This path satisfies the conditions required by \autoref{lem:saue} and hence it follows that $(\phi,\Iu)$ and $(\psi,\Iv)$ are strongly asymptotically unitarily equivalent.
\end{proof}

Complementary to \autoref{thm:original-KK-uniqueness}, we end this section by recording a number of immediate corollaries that follow from the main result in conjunction with the discussion in \autoref{ex:generalized-KK-examples}.

\begin{cor} \label{cor:dynamical-KK-uniqueness}
Let $A$ be a separable \cstar-algebra and $B$ a $\sigma$-unital stable \cstar-algebra.
Let $\alpha: G\curvearrowright A$ and $\beta: G\curvearrowright B$ be two actions of a second-countable, locally compact group.
Let $(\phi,\Iu), (\psi,\Iv): (A,\alpha)\to(\CM(B),\beta)$ be two absorbing cocycle representations that form an equivariant $(\alpha,\beta)$-Cuntz pair.
If $[(\phi,\Iu), (\psi,\Iv)]=0$ in $KK^G(\alpha,\beta)$, then $(\phi,\Iu)$ and $(\psi,\Iv)$ are strongly asymptotically unitarily equivalent. 
\end{cor}

\begin{cor} \label{cor:nuclear-KK-uniqueness}
Let $A$ be a separable \cstar-algebra and $B$ a $\sigma$-unital stable \cstar-algebra.
Let $\phi,\psi: A\to\CM(B)$ be two weakly nuclear and nuclearly absorbing $*$-homomorphisms that form a Cuntz pair.
If $[\phi,\psi]=0$ in $KK_{\mathrm{nuc}}(A,B)$, then $\phi$ and $\psi$ are strongly asymptotically unitarily equivalent. 
\end{cor}

\begin{cor} \label{cor:ideal-related-KK-uniqueness}
Let $X$ be a topological space.
Let $A$ be a separable $X$-\cstar-algebra and $B$ a $\sigma$-unital stable $X$-\cstar-algebra.
Let $\phi,\psi: A\to\CM(B)$ be two weakly $X$-equivariant $*$-homomorphisms that form a Cuntz pair.
\begin{enumerate}[label=\textup{(\roman*)},leftmargin=*]
	\item \label{cor:ideal-related-KK-uniqueness:1}
	Suppose that $\phi$ and $\psi$ are both absorbing elements in the set of weakly $X$-equivariant $*$-homomorphisms.
	If $[\phi,\psi]=0$ in $KK(X;A,B)$, then $\phi$ and $\psi$ are strongly asymptotically unitarily equivalent. 
	\item \label{cor:ideal-related-KK-uniqueness:2}
	Suppose that $\phi$ and $\psi$ are both weakly nuclear and absorbing elements in the set of weakly nuclear and weakly $X$-equivariant $*$-homom\-orphisms.
	If $[\phi,\psi]=0$ in $KK_{\mathrm{nuc}}(X;A,B)$, then $\phi$ and $\psi$ are strongly asymptotically unitarily equivalent. 
	\end{enumerate}
\end{cor}

%%%

\section{A generalized Paschke-type duality}

This section is devoted to proving a generalized Paschke-type duality for \cstar-dynamics that unifies earlier results of a similar nature.
We note, however, that the content of this section is vacuous for nonequivariant $KK$-theory, i.e., the case $G=\{1\}$, since it does not add anything beyond the statement of \autoref{thm:Paschke-1}.
We start with some preparatory observations.

\begin{lemma} \label{lem:D-phi-w-normcont}
Let $(\phi,\Iu): (A,\alpha)\to (\CM(B),\beta)$ be a cocycle representation and $x\in\SD^w_{(\phi,\Iu)}$.
Then for every $a\in A$, the map $g\mapsto\phi(a)\beta_g^\Iu(x)$ is norm-continuous.
\end{lemma}
\begin{proof}
Let $h\in G$ and $a\in A$.
We then have $\lim_{g\to h} \alpha_g(\alpha_h^{-1}(a))= a$ in norm and therefore
\[
\lim_{g\to h} \phi(a)\beta_g^\Iu(x) = \lim_{g\to h} \beta_g^\Iu(\phi(\alpha_h^{-1}(a))x).
\]
Since $a$ and $h$ were arbitrary and $\alpha$ is point-norm continuous on $A$, we see that the claim is equivalent to saying that the assignment $g\mapsto\beta^\Iu_g(\phi(a)x)$ is norm-continuous for every $a\in A$.
Let $\CQ(B)$ denote the corona algebra of $B$, let $\bar{\beta}^\Iu$ be the algebraic $G$-action induced by $\beta^\Iu$, let $\bar{\phi}: A\to\CQ(B)$ be the induced homomorphism, and write $\bar{x}=x+B\in\CQ(B)$.
By \cite[Theorem 2.1]{Brown00}, it suffices to show that the map $g\mapsto\bar{\beta}^\Iu_g(\bar{\phi}(a)\bar{x})$ is norm-continuous as a map into $\CQ(B)$.
The condition $[x,\phi(A)]\subseteq B$ means $\bar{x}\in\CQ(B)\cap\bar{\phi}(A)'$.
The condition $\phi(A)(\beta^\Iu_g(x)-x)\subseteq B$ for all $g\in G$ means that the element
\[
x^\sharp=\bar{x}+(\CQ(B)\cap\bar{\phi}(A)^\perp) \in (\CQ(B)\cap\bar{\phi}(A)')/(\CQ(B)\cap\bar{\phi}(A)^\perp)
\]
is in the fixed point algebra of the action induced by $\bar{\beta}^\Iu$, which we shall denote by $E$.
We have an induced (algebraic) $G$-equivariant $*$-homomor\-phism $\kappa: (A\otimes_{\max} E,\alpha\otimes\id)\to(\CQ(B),\bar{\beta}^\Iu)$ given by 
\[
\kappa\big( a\otimes \big( y+(\CQ(B)\cap\bar{\phi}(A)^\perp) \big) \big)= \bar{\phi}(a)y
\] 
for all $a\in A$ and $y\in\CQ(B)\cap\bar{\phi}(A)'$ with $y+(\CQ(B)\cap\bar{\phi}(A)^\perp)\in E$.
By equivariance of $\kappa$, it follows for all $g,h\in G$ that
\[
\begin{array}{ccl}
\|\bar{\beta}_g^\Iu(\bar{\phi}(a)\bar{x}) - \bar{\beta}_h^\Iu(\bar{\phi}(a)\bar{x})\| &=& \|\kappa\big( (\alpha_g(a)-\alpha_h(a) )\otimes x^\sharp \big) \| \\
&\leq& \| (\alpha_g(a)-\alpha_h(a))\otimes x^\sharp \| \\
&\leq& \|x\| \|\alpha_g(a)-\alpha_h(a)\|.
\end{array}
\]
This shows the desired continuity property and finishes the proof. 
\end{proof}

\begin{prop} \label{prop:same-quotients}
Let $(\phi,\Iu): (A,\alpha)\to(\CM(B),\beta)$ be a cocycle representation.
Keeping in mind \autoref{nota:D-phi} and \autoref{nota:C-phi}, one has $\SC_{(\phi,\Iu)}=\SD_{(\phi,\Iu)}\cap\SC_\phi$.
The resulting inclusion of quotients $\SB_{(\phi,\Iu)}=\SD_{(\phi,\Iu)}/\SC_{(\phi,\Iu)} \subseteq \SD^w_{(\phi,\Iu)}/\SC_{\phi}$ is an isomorphism.
\end{prop}
\begin{proof}
The only thing to show is that this map is surjective.
Concretely, we have to show that for every $x\in\SD^w_{(\phi,\Iu)}$, there exists $y\in\SD_{(\phi,\Iu)}$ such that $x-y\in\SC_{\phi}$.
Given $x$, we consider the strictly continuous map $\zeta: G\to\CM(B)$ given by $\zeta(g)=\beta^\Iu_g(x)-x$.
By definition of $\SD^w_{(\phi,\Iu)}$, we have $\zeta(G)\phi(A)\subseteq B$.
Furthermore, every map of the form $\zeta(\_ )\phi(a)$ is norm-continuous by \autoref{lem:D-phi-w-normcont}.
We apply Kasparov's technical theorem \cite[Theorem 1.4]{Kasparov88} for $B$ in place of $J$ with $G$-action $\beta^\Iu$, $\phi(A)$ in place of $A_1$, $\zeta$ in place of $\phi$, and $A_2=\Delta=0$.
This yields two positive contractions $M,N\in\CM^{\beta^\Iu}(B)$ with $M+N=\eins$ satisfying $\zeta(G)N\subset B$ and $\phi(A)M\subset B$.
Clearly $M,N\in\SD_{(\phi,\Iu)}$.
The element $y=xN\in\SD_\phi$ then satisfies for all $g\in G$ that
\[
\beta^\Iu_g(y)-y\equiv_B (\beta^\Iu_g(x)-x)N = \zeta(g)N \in B.
\]
Hence $y\in\SD_{(\phi,\Iu)}$.
On the other hand, we have for all $a\in A$ that
\[
(x-y)\phi(a)=x(1-N)\phi(a)=xM\phi(a)\in B.
\]
In other words, we see that indeed $x-y\in\SC_\phi$, which finishes the proof.
\end{proof}

The main idea behind the next lemma and its corollary is analogous to a similar argument in the proof of \cite[Lemma 6.4]{Thomsen05}.
We note that the assumption below, which we impose on a cocycle representation, is an analog of (but not identical to) what Thomsen therein refers to as a ``saturated'' equivariant $*$-homomorphism.

\begin{lemma} \label{lem:general-C-K-theory}
Let $(\phi,\Iu): (A,\alpha)\to(\CM(B),\beta)$ be a cocycle representation that weakly absorbs $(0,\Iu)$ and contains itself at infinity.
Then the inclusion $\SC_{(\phi,\Iu)}\subseteq \CM^{\beta^\Iu}(B)$ induces an embedding in $K$-theory with a split.
\end{lemma}
\begin{proof}
When comparing definitions, notice that $\SC_{(0,\Iu)}=\CM^{\beta^\Iu}(B)$.
Let $r_1,r_2\in\CM(B)^\beta$ be two isometries with $r_1r_1^*+r_2r_2^*=\eins$, which we use to form Cuntz sums.
By assumption, there exists a unitary $U\in\CM(B)$ such that $\ad(U)\circ(\phi\oplus 0,\Iu\oplus\Iu)$ agrees with $(\phi,\Iu)$ modulo $B$.
Consider the norm-continuous path of isometries given by 
\[
r_{1+t}:=\sqrt{1-t}\cdot r_1+\sqrt{t}\cdot r_2,\quad t\in [0,1].\footnote{Notice that this does not lead to conflicting notation.}
\]
Let us consider the point-norm continuous family of embeddings $\kappa_t: \SC_{(\phi,\Iu)}\to\CM(B)$ for $t\in [0,1]$ given by $\kappa_t(x)=r_{1+t} x r_{1+t}^*$.
Let $t\in [0,1]$ and $x\in\SC_{(\phi,\Iu)}$.
We notice for all $a\in A$ that
\[
\kappa_t(x)(\phi(a)\oplus 0)=r_{1+t} x\cdot \sqrt{1-t}\cdot \phi(a) r_1^* \in B.
\]
Futhermore, we have for all $g\in G$ that
\[
\begin{array}{ll}
\multicolumn{2}{l}{ \beta^{\Iu\oplus\Iu}_g(\kappa_t(x)) } \\
=& \beta^{\Iu\oplus\Iu}_g\Big( (1-t)r_1 x r_1^* + (t(1-t))^{1/2} \big( r_1xr_2^*+r_2xr_1^* \big) + t r_2 x r_2^* \Big) \\
=& (1-t)r_1 \beta_g^\Iu(x) r_1^* + (t(1-t))^{1/2} \big( r_1\beta_g^\Iu(x) r_2^*+r_2\beta_g^\Iu(x) r_1^* \big) + t r_2 \beta_g^\Iu(x) r_2^* \\
\equiv_B &  (1-t)r_1 x r_1^* + (t(1-t))^{1/2} \big( r_1 x r_2^*+r_2 x r_1^* \big) + t r_2 x r_2^* \\
=& \kappa_t(x).
\end{array}
\]
In other words, for each $t\in [0,1]$ the map $\kappa_t$ takes values in $\SC_{(\phi\oplus 0,\Iu\oplus\Iu)}$.
Let $\iota: \SC_{(\phi,\Iu)}\to \CM^{\beta^\Iu}(B)$ be the inclusion map.
Then the path $\{\kappa_t\}_{t\in [0,1]}$ witnesses the fact that the diagram
\[
\xymatrix{
\SC_{(\phi,\Iu)} \ar[rr]^\iota \ar[rd]_{\id\oplus 0} && \CM^{\beta^\Iu}(B) \ar[ld]^{0\oplus \id} \\
 & \SC_{(\phi\oplus 0,\Iu\oplus\Iu)} & 
}
\]
commutes up to homotopy.
The arrow $\id\oplus 0$ induces an isomorphism in $K$-theory by \autoref{lem:weak-absorption-K-theory}.
We can thus see that $K_*(\id\oplus 0)^{-1}\circ K_*(0\oplus\id)$ is a left inverse of $K_*(\iota)$.
This finishes the proof.
\end{proof}

\begin{cor} \label{cor:C-phi-vanishing-K-theory}
Let $\psi: A\to B$ be any $*$-homomorphism that contains itself at infinity and absorbs the zero representation.
Then $\SC_{\psi}$ has vanishing $K$-theory.
In particular, if $(\phi,\Iu): (A,\alpha)\to(\CM(B),\beta)$ is any cocycle representation such that $\phi$ contains itself at infinity and absorbs the zero representation, then the quotient map $\SD_{(\phi,\Iu)}^w\to\SB_{(\phi,\Iu)}$ (resulting from \autoref{prop:same-quotients}) induces an isomorphism in $K$-theory.
\end{cor}
\begin{proof}
Upon considering $G=\{1\}$, we may view $\psi$ as a cocycle representation in a trivial way.
The assumptions on $\psi$ translate into the assumptions required by \autoref{lem:general-C-K-theory}.
Hence the inclusion $\SC_{\psi}\subseteq\CM(B)$ is injective in $K$-theory.
The $K$-groups of $\CM(B)$ are known to vanish; see \cite[Lemma 3.2]{Cuntz81}.
This can also be deduced from applying \autoref{prop:infinite-repeat} to the special case $(\phi,\Iu)=(0,\eins)$.
The ``in particular'' part is an immediate consequence of the general properties of the $K$-theory functor.
\end{proof}

Below follows our more refined version of the Paschke duality theorem.
We note that for $G=\{1\}$, most of the terms simplify and then it just amounts to the statement of \autoref{thm:Paschke-1}, and we recover Thomsen's \cite[Theorem 3.2]{Thomsen01}.
His dynamical generalization \cite[Theorem 6.2]{Thomsen05} is very similar, but was phrased in the language of genuinely equivariant representations, whereas the result below is tailored to fit within the framework of cocycle representations.

\begin{theorem} \label{thm:Paschke-2}
Let $\alpha: G\curvearrowright A$ be an action on a separable \cstar-algebra and $\beta: G\curvearrowright B$ a strongly stable action on a $\sigma$-unital \cstar-algebra.
Let $\FC$ be an eligible set of cocycle representations $(A,\alpha)\to (\CM(B),\beta)$, and let $(\Phi,\IU)\in\FC$ be an absorbing element.
Consider the isomorphism $\Xi: K_1(\SD_{(\Phi,\IU)})\to \IE^G(\FC;\alpha,\beta)/_{\sim_{\FC,h}}$ from \autoref{thm:Paschke-1}.
Then $\Xi$ restricts to an isomorphism $K_1(\SC_{(\Phi,\IU)})\cong H_\beta$.
This induces isomorphisms $K_1(\SD_{(\Phi,\IU)}^w)\cong K_1(\SB_{(\Phi,\IU)})\cong KK^G(\FC;\alpha,\beta)$ such that the following diagram commutes:
\[
\xymatrix{
H_\beta \ar@{^{(}->}[r] & \IE^G(\FC;\alpha,\beta)/_{\sim_{\FC,h}} \ar@/^1.5pc/[l]_\Ff \ar@{->>}[r] & KK^G(\FC;\alpha,\beta) \\
K_1(\SC_{(\Phi,\IU)}) \ar[u]_{\cong} \ar[d] \ar@{^{(}->}[r] & K_1(\SD_{(\Phi,\IU)}) \ar[u]^{\Xi}_{\cong} \ar[ld] \ar@{->>}[r] & K_1(\SB_{(\Phi,\IU)}) \ar[u]_{\cong} \\
K_1(\CM^{\beta^\IU}(B)) \ar@/^4.0pc/[uu]_{\cong}^{\Xi_0} && K_1(\SD_{(\Phi,\IU)}^w) \ar[u]_{\cong}
}
\]
\end{theorem}
\begin{proof}
Before we begin with the proof, we point out a basic fact that we use below for convenience.
All of the \cstar-algebras appearing in the diagram above are either unital properly infinite \cstar-algebras or ideals in such a \cstar-algebra, hence any element in the $K_1$-group has a representative in the unitary group of the \cstar-algebra.
This makes it possible to avoid higher matrix amplifications below.

We begin by first explaining in what sense one should understand the leftmost upward arrow $\Xi_0$.
One has $\CM^{\beta^\IU}(B)=\SD_{(0,\IU)}$ upon comparing definitions.
Since $(\Phi,\IU)$ is an absorbing element of $\FC$ and $\FC$ contains all elements of the form $(0,\IV)$, it follows that $(\Phi,\IU)\asymp (\Phi\oplus 0,\IU\oplus\IV)$ for every $\beta$-cocycle $\IV$.
Since this implies $(0,\IU)\asymp ( 0,\IU\oplus\IV)$, it follows that $(0,\IU)$ is an absorbing cocycle representation when viewed as one with domain $(A,\alpha)=(\{0\},\id_{\{0\}})$.
Thus we may apply \autoref{thm:Paschke-1} in this sense and conclude that
\[
\Xi_0: K_1(\CM^{\beta^\IU}(B))\to H_\beta=\IE^G(\id_{\{0\}},\beta)/_{\sim_{h}}
\]
given by $[W]\mapsto\big[ \ad(W)\circ(0,\IU), (0,\IU) \big]$ is an isomorphism.

On the other hand, when we take $[U]\in K_1(\SD_{(\Phi,\IU)})$, then we see
\[
H_\beta\ni (\Ff\circ\Xi)([U])=\Ff\big[ \ad(U)\circ(\Phi,\IU), (\Phi,\IU) \big]=\big[ \ad(U)\circ(0,\IU), (0,\IU) \big].
\]
This means that the composition $\Ff\circ\Xi$ agrees with the composition of $K_1(\SD_{(\Phi,\IU)})\to K_1(\CM^{\beta^\IU}(B))\stackrel{\Xi_0}{\to} H_\beta$.

Now let $[W]\in K_1(\SC_{(\Phi,\IU)})$ for some $W\in\CU(\eins+\SC_{(\Phi,\IU)})$.
Let $r_1,r_2\in\CM(B)^\beta$ be two isometries with $r_1r_1^*+r_2r_2^*=\eins$, which we use to form Cuntz sums.
Exactly as in the proof of \autoref{lem:general-C-K-theory}, we consider the norm-continuous path of isometries $r_{1+t}:=\sqrt{1-t}\cdot r_1+\sqrt{t}\cdot r_2$ for $t\in [0,1]$ and the point-norm continuous family of embeddings $\kappa_t: \SC_{(\Phi,\IU)}\to\SC_{(\Phi\oplus 0,\IU\oplus\IU)}$ for $t\in [0,1]$ given by $\kappa_t(x)=r_{1+t} x r_{1+t}^*$.

The assignment $[0,1]\ni t\mapsto r_t W r_t^*+(\eins-r_tr_t^*)$ induces a homotopy between $W\oplus\eins$ and $\eins\oplus W$ inside $\CU(\eins+\SC_{(\Phi\oplus 0,\IU\oplus\IU)})$.
This yields a chain of $\FC$-homotopies
\[
\begin{array}{ll}
\multicolumn{2}{l}{ \big( \ad(W)\circ(\Phi,\IU), (\Phi,\IU) \big) }\\
\sim_{h,\FC}& \big( \ad(W\oplus\eins)\circ(\Phi\oplus 0,\IU\oplus\IU), (\Phi\oplus 0,\IU\oplus\IU) \big) \\
\sim_{h,\FC}& \big( \ad(\eins\oplus W)\circ(\Phi\oplus 0,\IU\oplus\IU), (\Phi\oplus 0,\IU\oplus\IU) \big) \\
\sim_{h,\FC}& \big( \ad(W)\circ(0,\IU), (0,\IU) \big).
\end{array}
\]
Hence the composition of the maps $K_1(\SC_{(\Phi,\IU)})\to K_1(\CM^{\beta^\IU}(B))\stackrel{\Xi}{\to} H_\beta\to \IE^G(\FC;\alpha,\beta)/_{\sim_{\FC,h}}$ is equal to the composition of the maps $K_1(\SC_{(\Phi,\IU)})\to K_1(\SD_{(\Phi,\IU)})\stackrel{\Xi}{\to}\IE^G(\FC;\alpha,\beta)/_{\sim_{\FC,h}}$.
This induces the homomor\-phism
\[
K_1(\SC_{(\Phi,\IU)})\to H_\beta
\]
such that the whole left half of the diagram commutes.
It is injective by \autoref{lem:general-C-K-theory}.
We claim that it is also surjective.
Given any $x\in H_\beta$, we may appeal to the surjectivity of $\Xi_0$ and find a unitary $V\in\CM^{\beta^\IU}$ such that $x=\big[ \ad(V)\circ(0,\IU), (0,\IU) \big]$.
Since the map 
\[
K_1(\id\oplus 0): K_1(\SC_{(\Phi,\IU)})\to K_1(\SC_{(\Phi\oplus 0,\IU\oplus\IU)})
\] 
is an isomorphism by \autoref{lem:weak-absorption-K-theory}, we find a unitary $W\in\CU(\eins+\SC_{(\Phi,\IU)})$ with $[W\oplus\eins]=[\eins\oplus V]$ in $K_1(\SC_{(\Phi\oplus 0,\IU\oplus\IU)})$.
By the $K_1$-injectivity of $\SD_{(\Phi\oplus 0,\IU\oplus\IU)}$, we have that $W\oplus\eins$ and $\eins\oplus V$ are homotopic inside the unitary group of $\SD_{(\Phi\oplus 0,\IU\oplus\IU)}$.
The analogous computation to the above then shows that
\[
\big[ \ad(W)\circ(\Phi,\IU), (0,\IU) \big]=\big[ \ad(V)\circ(0,\IU), (0,\IU) \big]=x.
\]
This takes care of the left half of the diagram.

The isomorphism $K_1(\SB_{(\Phi,\IU)})\to KK^G(\FC;\alpha,\beta)$ making the rest of the diagram commute is then uniquely induced because 
\[
K_1(\SC_{(\Phi,\IU)})\hookrightarrow K_1(\SD_{(\Phi,\IU)})\twoheadrightarrow K_1(\SB_{(\Phi,\IU)})
\]
is also a short exact sequence due to $\Ff$ inducing a left inverse for the first arrow.
We have already observed in \autoref{prop:same-quotients} and \autoref{cor:C-phi-vanishing-K-theory} that the arrow $K_1(\SD^w_{(\Phi,\IU)})\to K_1(\SB_{(\Phi,\IU)})$ is a well-defined isomorphism.
This finishes the proof.
\end{proof}

\begin{rem}
Let $\FC$ be an eligible set of cocycle representations $(A,\alpha)\to(\CM(B),\beta)$ and let $(\Phi,\IU)\in\FC$ be an absorbing element.
The isomorphism $K_1(\SB_{(\Phi,\IU)})\cong KK^G(\FC;\alpha,\beta)$ from \autoref{thm:Paschke-2} is, strictly speaking, only the first of two possible versions of Paschke duality in our context.
It can be viewed as a generalization of \cite[Theorem 6.2]{Thomsen05} in the language of cocycle representations instead of genuinely equivariant ones.
The second version of Paschke duality should relate $K_0(\SB_{(\Phi,\IU)})$ to a generalized version of the group $KK^G_1(\alpha,\beta)$ and/or $\operatorname{Ext}^G(\alpha,\beta)$ built from cocycle representations or Busby maps related to $\FC$.
A rigorous investigation of how such generalized equivariant Ext- or $KK_1$-groups would actually look like is beyond the scope of this article, but the apparent template for this kind of second Paschke duality would be \cite[Theorems 6.3, 6.5]{Thomsen05}; see also \cite[Section 3]{Thomsen01}.
For $\FC$ being the set of all cocycle representations, the related constructions are available in the literature, and such a result could in this case perhaps be achieved with some extra work if one appeals to \cite[Lemma 4.4]{Thomsen00}.
Since that result and its proof appeal to various technical results in the Kasparov picture of $KK^G_1$ (including the general form of the Kasparov product and its properties from \cite{Kasparov81}), it is not immediately clear how to generalize the argument to the present context in a self-contained manner.
\end{rem}

%%%%%%%%%%%%%%%%%%%%%%%%%%%%%%%%%%%%%%%%%%%%%%%%%%%%%%%%%%%%%%%%%%

%\newpage

\subsection*{Acknowledgements} \addcontentsline{toc}{section}{Acknowledgements}
The author was supported by the European Research Council under the European Union's Horizon Europe research and innovation programme (ERC grant AMEN--101124789).
For the purpose of open access, the author applies a CC BY public copyright license to any author accepted manuscript version arising from the submission of this work.

I would like to thank Stuart White, Shanshan Hua, James Gabe, Sergio Girón Pacheco and Robert Neagu for fruitful conversations regarding the topic of this work, and for providing feedback on earlier versions of the article.
This work has benefited from participation in the programme \emph{Topological groupoids and their \cstar-algebras} in July 2025 and I thank the Isaac Newton Institute for Mathematical Sciences, Cambridge, for the support and hospitality.
That programme was supported by EPSRC grant EP/V521929/1.

%%%%%%%%%%%%%%%%%%%%%%%%%%%%%%%%%%%%%%%%%%%%%%%%%%%%%%%%%%%%%%%%%%

\end{document}